%% file: main.tex
\newif \ifHAL
\HALfalse
\PassOptionsToPackage{dvipsnames}{xcolor}

\ifHAL
\documentclass[12pt]{article}
\usepackage{geometry}
\else
\documentclass[12pt,leqno]{siamart251104}
\fi

\usepackage[T1]{fontenc}
\usepackage[utf8]{inputenc}
\usepackage[american]{babel}
\usepackage{color}
\usepackage{amssymb}
\ifHAL
\usepackage{amsmath,amsthm}
\fi
\usepackage{siunitx}
\usepackage{graphicx}
\usepackage{subcaption}
\usepackage{placeins}
\usepackage{xspace}
\usepackage{bm,bbm}
\usepackage{centernot}
\usepackage[numbers,sort]{natbib}
\let\cite=\citet
\usepackage{enumerate,array}
\usepackage{dsfont}
\usepackage{enumitem}
\setlist{leftmargin=20pt}
\usepackage{Macros/mydef}
\usepackage{geometry}
\usepackage{comment}
\usepackage{multirow}
\usepackage{diagbox}

\usepackage{algorithm}
\usepackage{algpseudocode}

\begin{document}
\newcommand\footnotemarkfromtitle[1]{%
\renewcommand{\thefootnote}{\fnsymbol{footnote}}%
\footnotemark[#1]%
\renewcommand{\thefootnote}{\arabic{footnote}}}

\title{Locally conservative redistribution limiting and
  applications to the approximation of conservation equations, Part II}

\author{
  Jean-Luc~Guermond\footnotemark[1]
}

\date{Draft version \today}

\maketitle

\renewcommand{\thefootnote}{\fnsymbol{footnote}}

\footnotetext[1]{Department of Mathematics, Texas A\&M University 3368 TAMU, College
  Station, TX 77843, USA.}%

\footnotetext[2]{This material is based upon work supported in part
  by the National Science Foundation grant DMS2110868,
  the Air Force Office of Scientific Research, USAF, under grant/contract number FA9550-18-1-0397,
  the Army Research Office, under grant number W911NF-19-1-0431,
  and the U.S. Department of Energy by Lawrence Livermore National
  Laboratory under Contracts B640889.}

\renewcommand{\thefootnote}{\arabic{footnote}}
\begin{abstract}
 A non-intrusive, discretization and PDE agnostic limiting technique is
 introduced. The method is conservative and can be applied to
 high-order approximations including spectral methods. It can be
 applied to time-dependent and time-independent problems. It can be used
 for elliptic, parabolic and hyperbolic problems.
\end{abstract}

\ifHAL


\else


\fi

\pagestyle{myheadings} \thispagestyle{plain}
\markboth{}{}


\section{Introduction}

This paper is the second part of a work started in
\cite{Guermond_Wang_JCP_2025} about limiting.  It is also an extension
of the ``repair'' method introduced in
\cite[\S2.3]{Kucharik_Shashkov_Wendroff_JCP_2003} and the ``sweeping
method'' proposed in
\cite{Liu_Cheng_Shu_JSC_2017,Griffin_Shu_RMathSci_2026}.

\subsection{Objectives and content of the paper}
This work is part of a long term research project started in 2008 with
\cite{Guermond_Pasquetti_CRASS_2008,Guermond_Pasquetti_Popov_2011}
about the approximation of nonlinear conservation equations using
spectral and high-order spectral element methods.  The question
addressed in
\citep{Kucharik_Shashkov_Wendroff_JCP_2003,Liu_Cheng_Shu_JSC_2017,Guermond_Wang_JCP_2025,Griffin_Shu_RMathSci_2026}
is about limiting the approximations of the solutions to conservation
conservation equations without invoking squeezing techniques akin to
the Flux-Transport-Correction-inspired (FCT) (see
\cite{Boris_Book_JCP_1973,Zalesak_1979}).  Here, by conservation
equations, we mean partial differential equations (PDEs) endowed with
conservation properties.  These could be steady-state PDEs (hyperbolic
or elliptic), time-dependent hyperbolic or parabolic PDEs, or PDEs of
mixed types. The nature of the PDEs or approximations thereof does not
really matter as what is done in
\citep{Liu_Cheng_Shu_JSC_2017,Guermond_Wang_JCP_2025,Griffin_Shu_RMathSci_2026}
and what is proposed in this paper is essentially a consistent
conservative non intrusive post-processing methodology of whichever
approximation is available.

Recall that in the FCT paradigm a ``non-admissible'' but accurate
high-order solution is squeezed in the direction of a ``good''
low-order solution, while maintaining conservation, until the squeezed
high-order solution is in some targeted admissible set. The squeezing
idea has been used with immense success in the context of the
discontinuous Galerkin approximation of nonlinear conservation
equations, see \eg \cite{Zhang_Shu_JCP_2010,Zhang_Shu_ProcA_2011}, and
in the context of finite volume literature, see \eg
\cite{Sanders_1988}, \cite{Liu_Tadmor_1998}, \cite{KPP_2007}.  But the
FCT strategy fails when no ``good'' and ``robust'' low-order solution
is available, or when the high-order and low-order approximations do
not carry exactly the same ``generalized'' mass, or when the stencil
of the high-order approximation is very large. The purpose of the
conservative redistribution limiting methodologies introduced in
\citep{Kucharik_Shashkov_Wendroff_JCP_2003,Liu_Cheng_Shu_JSC_2017,Guermond_Wang_JCP_2025,Griffin_Shu_RMathSci_2026}
is to entirely bypass the high-order-to-low-order squeezing paradigm
(henceforth called FCT paradigm) to avoid these problems. Since
contrary to FCT-inspired squeezing techniques the proposed new methods
do no rely on the existence of a low-order approximation that carries
exactly the same generalized mass, they can be applied to high-order
approximations upon using an appropriate change of basis.  The key
novelties of the present work are threefold: (1) We extend
\citep{Liu_Cheng_Shu_JSC_2017,Guermond_Wang_JCP_2025,Griffin_Shu_RMathSci_2026}
to systems, \ie the new proposed method accounts for state variables
taking values in $\Real^q$ with $q\ge 1$; (2) We propose a variation
of the method that is cell-based and therefore may be better
adapted to high-performance computing (using vectorization, GPUs, or
matrix free softwares for instance); (3) We show (and confirm the
observation made in \citep[\S3]{Liu_Cheng_Shu_JSC_2017}) that the
method can be applied to spectral methods, which was a question we
have been working on and off since the 2008 paper
\citep{Guermond_Pasquetti_CRASS_2008} without great success until now.

The paper is organized as follows. Notation and key preliminary
details are introduced in \S\ref{Sec:setting_for_abstract_limiting}.
The generic abstract algorithm mentioned above is introduced in
\S\ref{Sec:Jacobi}. This algorithm is not specific to any particular
discretization method. The main results of this section are
Lemma~\ref{Lem:cell.local_general_limiting} and
Lemma~\ref{Lem:cell.global_general_limiting}.  Examples of discrete
setting satisfying the assumptions of the generic abstract algorithm
are constructed in \S\ref{Sec:Examples_of_discretzation_settings}.
Numerical illustrations are reported in \S\ref{Sec:numerical
  illustrations}.

\subsection{Literature context for the present work}

The main problem with techniques based on the squeezing idea of the
Flux Transport Correction by \cite{Boris_Book_JCP_1973,Zalesak_1979}
(and generalization thereof for systems, like convex limiting by
\citep{Guermond_Nazarov_Popov_Tomas_2018,Guermond_Popov_Tomas_CMAME_2019})
is that these methods require to have at hand an invariant-domain
preserving (IDP) low order method that carries exactly the same mass
(on average) as the high-order one. While the requirement of having
two methods that carry the same mass is not a severe constraint when
using piecewise linear approximations, it becomes cumbersome when
using high-order elements, and it is even a dead end when one wants to
use spectral methods (\eg Fourier approximation). The reader is also
referred \cite{TurekKuzmin2002}, \cite{KuzminLoehnerTurek2012}, and
the literature therein for other continuous finite element extensions
of the FCT idea. Two key difficulties with the FCT strategy have been
identified in the PhD theses of \cite[\S3.3.2.1]{Fahad_2015}, defended
in 2015, and \cite[\S4.2.2]{Quezada_2016}, defended in 2016, which I
supervised.
(i) The first difficulty is that low-order
invariant-domain preserving methods based on high-order polynomial
degree approximations are not robust with respect to the polynomial
degree.  The bound on the Courant–Friedrichs–Lewy (CFL) number that is
required to maintain the invariant-domain preserving property
decreases very fast as the polynomial degree increases;
(ii) The
second difficulty is that low-order IDP methods are more and more
diffusive as the polynomial degree increases because the connectivity
in the stencil grows like $k^d$, where $k$ is the polynomial degree
and $d$ the space dimension. These two observations are also reported
in \cite[\S3.3]{Anderson_etal_2017}, although it seems that Bernstein
finite elements behave a bit better than Lagrange elements in this
respect as observed in \citet[S4.2.2]{Quezada_2016} and
\citet[\S3.3]{Anderson_etal_2017}.
(iii) Let me add also that FCT-like strategies are highly intrusive in
that they require that the user has access to the low-order and
high-order fluxes in order to mix them. Hence, FCT-like methods are
neither approximation agnostic nor performance portable.

In the wake of the two PhD theses of \citet{Fahad_2015} and
\citet{Quezada_2016}, it was clear that the problems mentioned above
(i)--(iii) could in principle be solved by invoking some hierarchical
decomposition of the space approximation. But progresses toward
constructing a hierarchical method having the right properties have
been slow. The main roadblock on the way was that we insisted that the
low-order method and the high-order method carry exactly the same
generalized mass in order to be able to apply the Zalesak paradigm
(\ie FCT or variations thereof).  A solution to this conundrum for
fourth-order finite differences has proposed in \cite{ErnGu:21+} but
generalization to non-uniform finite elements was not straightforward.
An original idea in this direction in the context of high-order
continuous finite element residual distribution schemes has been
proposed in a breakthrough paper by
\cite{Abgrall_Viville_Beaugendre_Dobrzynski_2017}.  In the wake of
\citep{Abgrall_Viville_Beaugendre_Dobrzynski_2017}, an IDP method
using high-order continuous finite elements has been proposed in
\citep{Guermond_Nazarov_Popov_CMAME_2024}. The first key idea in
\citep{Guermond_Nazarov_Popov_CAMME_2024}, is to optimize the CFL number by
restricting the low-order method to having a stencil only based on the
next neighbors. The second key idea is to slightly modify the fluxes
of the low-order method so that it carries exactly the same
generalized mass as the high-order method while still being
conservative and consistent (as in
\citep{Abgrall_Viville_Beaugendre_Dobrzynski_2017}). The proposed
technique is robust with respect to the polynomial degree and is
exactly conservative on patches. But, the construction of the modified
fluxes is technical and the algorithmic complexity (understandably)
grows like $k^d$. In any case, we do not see how to extend this idea
to spectral methods. Note in passing that the idea of constructing an
IDP low-order method with modified fluxed in the spirit of
\citep{Abgrall_Viville_Beaugendre_Dobrzynski_2017} has also been
explored with remarkable success in the context of discontinuous
finite element by \cite{Pazner_2021}.

In the light of the argumentation above, and in the continuation of
\citep{Guermond_Wang_JCP_2025} and as in
\citep{Kucharik_Shashkov_Wendroff_JCP_2003,Liu_Cheng_Shu_JSC_2017,Griffin_Shu_RMathSci_2026}
we now adopt a new point of view that simply consists of completely
abandoning Zalesak's squeezing paradigm.  Limiting is simply done by
invoking locally conservative redistribution of the states (or
conservative redistribution limiting; CRL for short).  This idea has
seemingly been pioneered in
\cite[\S2.3]{Kucharik_Shashkov_Wendroff_JCP_2003} in the context of
arbitrary Lagrangian–Eulerian method to limit the density using
specific information from the underlying mesh. It has been branded
``repair'' therein. The ``repair'' method has been analyzed in
\cite{Shashkov_Wendroff_JCP_2004} still focusing on the Euler
equations and arbitrary Lagrangian–Eulerian methods. The convergence
of the method has been analyzed in
\cite{Despres_Loubere_IJ_Finite_Vol_2006} when applied to the
one-dimensional linear transport equation.  The conservative
redistribution limiting idea has been explored in
\cite[\S3.2.1]{Laiu_Hauck_JSC_2019} in the context of the
time-dependent state neutron transport equation using the sweeping
technique of \citep{Liu_Cheng_Shu_JSC_2017} and in
\citep{Guermond_Wang_JCP_2025} in the context of the steady state
neutron transport equation, and in both cases CRL has been
demonstrated to be efficient and robust.  The method has been shown in
\citep{Liu_Cheng_Shu_JSC_2017,Griffin_Shu_RMathSci_2026} to be
effective to limit explicit approximations of the compressible Euler
equations using various approximation techniques (including Fourier).
Notice finally that as originally formulated the ``repair'' method and
the sweeping methods were sequential (\ie could not be parallelized or
vectorized) and would depend on the enumeration that is used to
traverse the states, see \eg \cite[\S5.3]{Laiu_Hauck_JSC_2019},
\cite[\S2.3]{Griffin_Shu_RMathSci_2026} and
\cite[\S5.3]{Loubere_Staley_Wendroff_JCP_2006}. A fix for the
``repair'' method consisting of applying a few times a local CRL steps
and applying one global CRL step is proposed in
\cite[\S6.1]{Loubere_Staley_Wendroff_JCP_2006}. We adopt this idea
latter in the paper.

Notice finally that conservative redistribution limiting not using the
squeezing paradigm can be done by invoking nonlinear optimization like
in \eg \cite{Rider_Kothe_AIAA_1997},
\cite{Bochev_Ridzal_Peterson_JCP_2014},
\cite{LiuRiviereShenZhang_SIAM_JSC_2024},
\cite{Liu_Milesis_Shu_Zhang_JCP_2026}. The approach adopted in the
present paper is to avoid resorting to constrained optimization or
convex minimization. The price paid though for this choice is that the
local conservative redistribution limiting technique proposed in the
paper is not guaranteed to converge in a fixed number of iterations
independent of the cardinality of states to be limited. Guaranteed
limiting is obtained through a global conservative redistribution
limiting as in \cite[\S6.1]{Loubere_Staley_Wendroff_JCP_2006}.

The key contributions of the paper are as follows:
(1) We propose a generic abstract Conservative Redistribution Limiting
(CRL) algorithm.
(2) We describe a cell-based variation of the method that is well adapted
to high-performance computing as it only involves cell-wise stencils.
(3) We describe how the method can be used in the context of very
high-order methods using both modal and nodal representations of the
approximation (spectral elements or Fourier approximation).
(4) The performance of the proposed method is numerically illustrated.

\section{Iterative limiting algorithm: notation, heuristics}
\label{Sec:setting_for_abstract_limiting}

Since the algorithms presented in the paper are not specific to any
particular discretization or system of conservation equations, we
start by introducing some abstract notation to reflect this
generality.

\subsection{The convex invariant set and heuristics} 
Let $q$ be a natural number larger than or equal to $1$. We assume to
be given a finite collection of states (\ie vectors) in $\Real^q$,
which we call unlimited numerical states.  These states may be the
outputs of some algorithm solving some nonlinear conservation equation
(whether these are time-dependent or not does not matter) or some PDE
on integro-differential equation for which some notion of conservation
holds and some invariant domain property can be identified.  The key
assumption of the entire paper is that the solutions to the system of
nonlinear equations in question take values in a convex set $\calA$ in
$\Real^q$, which we call invariant domain, and a notion of
conservation holds.

Our objective is to devise a conservative post-processing technique
that maps the given unlimited numerical states to the invariant domain
$\calA$ if there are not already there. One important idea that makes
the whole enterprise reasonable and doable is that we assume that the
unlimited numerical states are consistent approximations of exact
states that are actually members of $\calA$.  By consistent, we mean
that the unlimited numerical states are already almost all in $\calA$
or are very close to the boundary of $\calA$.  This is indeed
naturally the case if the unlimited numerical states are computed by
using some time stepping strategy with some time step $\dt$ small
enough using state data that are already in $\calA$. In this case, the
distance of the numerical states to the boundary of $\calA$ cannot be
larger than $\calO(\dt)$. In principle, limiting the numerical states
should not destroy the approximation properties of the approximation
method since it is a consistent operation, \ie the exact states are
members of $\calA$ after all.

\subsection{Enumeration}\label{Sec:Enumeration}

Let us assume to be given a collection of states in $\Real^q$
enumerated over some index set $\calV$, say $\{\bsfu_i\}_{i\in\calV}$
with $\bu_i\in\Real^q$ for all $i\in\calV$.  We henceforth say that
$\calV$ is the vertex index set and the states in
$\{\bsfu_i\}_{i\in\calV}$ are the unlimited states.  To formalize the
conservation properties alluded to above, which we recall is a central
notion in the proposed algorithm, we assume to be given a collection
of nonnegative real numbers $\{m_i\}_{i\in\calV}$ which we call lumped
masses coefficients.  We call total ``generalized mass'' of the state
$\{\bsfu_i\}_{i\in\calV}$ the $\Real^q$-valued vector $\bcalM \eqq
\sum_{i\in\calV} m_i \bsfu_i$. Notice that $\calM$ is
$\Real^q$-valued. Our objective is to construct a post-processing
technique for the unlimited states while maintaining $\bcalM$
constant.

We also assume to be given a subset $\calV(\calT)$ of the power set
$\calP(\calV)$.  Here $\calT$ is the index set that is used to
enumerate the members of $\calV(\calT)$. For every index $K$ in the
index set $\calT$ we use the notation $\calV(K)$ instead of $\calV_K$
to denote the member of $\calV(\calT)$ with index $K$. The subset
$\calV(\calT)\subset \calP(\calV)$ is then represented as follows:
$\calV(\calT)=\bigcup_{K\in\calT}\calV(K)$, where for all $K\in\calT$
the subset $\calV(K)\subset\calV$ is the $K$-th member of
$\calV(\calT)$. Recall that $\calV(K)$ is a collection of indices in
$\calV$. We henceforth call the members of
$\calT$ cell indices, and we call $K$ cell index (or cell for
short).

The definition of the set $\calV(\calT)$ allows us to define a second
subset of the power set $\calP(\calT)$ that we enumerate with the
index set $\calV$ and that we denote $\calT(\calV)$.  For every index
$i$ in the index set $\calV$ we use the notation $\calT(i)$ instead of
$\calT_i$ to denote the member of $\calT(\calV)$ with index $i$. The
subset $\calT(\calV)\subset \calP(\calT)$ is then represented as
follows: $\calT(\calV)=\bigcup_{i\in\calV}\calT(i)$. The set $\calP(\calT)$ is defined
so that the following holds true for all $i$
in $\calV$ and all $K\in\calT$: 
\begin{equation}
  \left(K\in \calT(i)\right) \Longleftrightarrow
\left(i\in \calV(K)\right). \label{KinTi_iif_iinVK}
\end{equation}

To be able to localize the limiting operation over cells, and
therefore maintain conservation locally, we further refine the setting
by assuming that the notion of lumped mass can be redistributed over
cells. More precisely, for all $i\in\calV$ we assume that there exists
a collection of nonnegative numbers $\{m_{i,K}\}_{K\in\calT(i)}$ so
that%
\begin{equation}\label{def_mik}%
  m_i = \sum_{K\in\calT(i)} m_{i,K}\quad \text{with}\quad m_{i,K}\ge 0 \quad \forall i\in\calV,
  \forall K\in\calT(i).%
\end{equation}
The definitions 
\eqref{KinTi_iif_iinVK} and \eqref{def_mik} imply that the
following two identities hold:
\begin{equation}
\sum_{K\in\calT} \sum_{i\in \calV(K)}
 = \sum_{i\in\calV} \sum_{K\in\calT(i)} ,
 \quad\text{and}\quad \sum_{i\in\calV} \sum_{K\in\calT(i)} m_{i,K} \bu_i=\sum_{i\in\calV} m_{i} \bu_i.
 \label{Fubini}
\end{equation}
Examples of enumerations in various contexts are collected in
\S\ref{Sec:Examples_of_discretzation_settings}.

\subsection{Limiting}  \label{Sec:abstract_limiting}

Our objective is to find a conservative limiting technique for the
states $\{\bu_i\}_{i\in\calV}$. More precisely, given a convex set
$\calA$ in $\Real^q$,
 and denoting $I\eqq \card(\calV)$, we want to construct
an algorithm with $\calO(I)$ complexity that maps the
states $\{\bu_i\}_{i\in\calV}$ from $\Real^q$ to $\calA$ and that is
conservative. In other words we want to find a (nonlinear) mapping
\begin{equation} \label{conservative_limiting}
  \calL:(\Real^q)^I \to \calA^I\qquad
  \text{so that}\qquad \sum_{i\in\calV} m_i \bu_i = \sum_{i\in\calV} m_i (\calL(\bu))_i.
\end{equation}

In many applications, limiting may not refer to one process only but
may be instead refer to a sequence of processes.  For instance, for scalar
conservation equations, one may want to enforce first the minimum
principle, then enforce the maximum principle, then finally enforce
some local entropy inequality (one may even consider more than one
entropy).  For the compressible Euler equations, one may want to make
sure that the density is positive; then make sure that the density is
below the maximal compressibility threshold if a co-volume equation of state is
used; then make sure that the entropy satisfies the minimum principle;
then possibly enforce some maximum local bound on the kinetic energy
or any over relevant and consistent constraint. To address the
situation described above, we will take as a template the general
limiting strategy described in Remark~7.23 in
\citep{Guermond_Popov_Tomas_CMAME_2019}; see also \S3.1 in
\citep{Guermond_Maier_popov_Saavedra_Tomas_JSC_2024}. In particular,
we assume that limiting is done by using quasiconcave functions (the
reader who is not familiar with the term may replace it by concave
functions).

\begin{definition}[Quasiconcavity]\label{def:quasiconv} Given a convex set
  $\calC \subset \Real^q$, we say that a function $\Psi:\calC \to \Real$ is
  quasiconcave if the set $L_\chi(\Psi) := \{\bu\in \calC \st \Psi(\bu)
  \ge \chi \}$ is convex for every $\chi\in \Real$. The sets
  $\{L_\chi(\Psi)\}_{\chi\in \Real}$ are called \emph{upper level sets} or
  \emph{upper contour sets}.
\end{definition}

We now list the properties we are interested in and that we are going to use to limit
the sates $\{\bu_i\}_{\calV}$. We assume that there is a natural number $L\ge 1$,  
a collection of $L+1$ subsets $\{\calB_l\}_{l\in\intset{0}{L}}$ in
$\Real^q$, and a collection of $L$ continuous quasiconcave functionals
$\{\Psi_l\}_{l\in\intset{1}{L}}$ so that the following
properties hold true:%
\begin{subequations}
  \label{Ass_Psi}%
  \begin{alignat}{2}%
    &\calB_{L}\subset \calB_{L-1}\subset\ldots\subset \calB_0\eqq \Real^q, \qquad && \label{Ass_Psi:1}
    \\
    &\Psi_l:\calB_{l-1} \to \Real  && \forall l\in \intset{1}{L}\\
    &\calB_{l}=\{ \bu\in \calB_{l-1}\st \Psi_{l}(\bu)\ge 0\}, &&  \forall l\in\intset{1}{L},\label{Ass_Psi:2}
    \\
    &\calB_{L} \subset \calA, && 
    \label{Ass_Psi:3}
  \end{alignat}
\end{subequations}
Notice in passing that $\calB_0=\Real^q$ is convex by construction.
Moreover, the definition $\calB_{l} = L_0(\Psi_l)$ for all
$l\in\intset{1}{L}$ implies that $\calB_{l}$ is the upper 0-level sets
of $\Psi_l$ and $\calB_{l}$ is therefore convex. Hence all the subsets
$\{\calB_l\}_{l\in\intset{0}{L}}$ are convex. All the sets
$\{\calB_l\}_{l\in\intset{1}{L}}$ are also closed as the functional
$\{\Psi_l\}_{l\in\intset{1}{L}}$ are continuous.

Finally, as it has been well established in the literature that
limiting is most effective when localized (see \eg
\cite{Kolgan_JCP2011_1972} \cite{Boris_Book_JCP_1973},
\cite{VANLEER_JCP_1974}, \cite{Zalesak_1979}).
Hence, we furthermore assume to have at hand a local version of
$\Psi_l$ for all $i\in\calV$ which we denote $\Psi_l^i:\calB_{l-1} \to
\Real$.  To fix the idea and facilitate the reading of the paper, it may be
helpful to the reader to think of the following example:
\begin{equation}
  \Psi_l^i(\bv) \eqq \Psi_l(\bv) - \Psi_{l}^{i,\min}
  \quad \text{with}\quad \Psi_{l}^{i,\min}\ge 0\quad \forall i\in\calV,
  \label{def_of_Psi_i}
\end{equation}
where $\Psi_{l}^{i,\min}$ is a known a priori local bound for the
$i$-th state, $i\in\calV$, for the $l$-th functional. Notice in
passing that this assumption is such that enforcing $\Psi_l(\bv) \ge
\Psi_{l}^{i,\min}$ implies $\Psi_l^i(\bv)\ge 0$, meaning that
$\bv\in\calB_l$.  The underlying idea here is that enforcing many good
local bounds is better than enforcing one global bound. All the
numerical tests reported in the paper are done with this type of local
functional. How the numbers $\Psi_{l}^{i,\min}$ are constructed is
explained latter, but this construction is irrelevant at the moment.

\section{Jacobi Cell limiting} \label{Sec:Jacobi}

We propose a Jacobi-like algorithm based on a twofold strategy: we
first balance the bounds locally on each cell, then average the
local limited state at the end of the loop over the cells. The
operations done on each cell are independent of the operations done on
the other cells. This approach solves the issues with parallelism
identified for ``repair'' methods in
\cite[\&5.3]{Loubere_Staley_Wendroff_JCP_2006}.

\subsection{The algorithm}
Let us now describe the algorithm. We assume to be at the $l$-th stage
of the limiting process as defined in \eqref{Ass_Psi}, where
$l\in\intset{1}{L}$. That is to say we assume that the unlimited
numerical states are all in $\calB_{l-1}$, \ie
\begin{equation}
\Psi_{l-1}(\bu_i)\ge 0\qquad \forall i\in\calV.
\end{equation}
The next task at hand consists of constructing a conservative
(nonlinear) limiting process $\calL$ so that%
\begin{equation}
  \Psi_l((\calL(\bu))_i)\ge 0 \quad \forall i\in\calV.
\end{equation}
The symbol $\calL$ should be indexed with $l$, but we refrain doing so
to simplify the notation.  We now drop the index $l$ when
the context is unambiguous (but revive it when it is useful).

The proposed algorithm is a loop over the cells.
For each cell index $K$ in $\calT$, the algorithm consists
of performing the following steps:\medskip

\newStep

\step{cell.lim.local.mass.step1} Recalling that 
$\Psi^i(\bv)$ is a localized, possibly shifted, version of $\Psi(\bv)$, the index set
$\calV(K)$ is partitioned into three sets $\calV(K)^+$, $\calV(K)^-$,
and $\calV(K)^0$:%
\begin{subequations}\label{def_of_the_Vks_global_lumped_mass}
  \begin{align}
  \calV(K)^+ &\eqq \{i\in\calV(K) \st \Psi^i(\bu_i) > 0\},\\
  \calV(K)^- &\eqq \{i\in\calV(K) \st \Psi^i(\bu_i) < 0\},\\
  \calV(K)^0 &\eqq \{i\in\calV(K) \st \Psi^i(\bu_i) = 0\}.
  \end{align}
\end{subequations}
Notice that the above partition of $\calV(K)$ depends on the states
$\{\bu_i\}_{i\in\calV}$ and we should probably use a super-index $\bu$
to recall this dependence, but to simplify the notation we simply
write $\calV(K)^+$, $\calV(K)^-$, and $\calV(K)^0$ instead of
$\calV^\bu(K)^+$, $\calV^\bu(K)^-$, and $\calV^\bu(K)^0$ since the
context is unambiguous.

If it happens that $\calV(K)^-$ is empty, \ie all the states are in
bonds, then we should set $\calL(\bu)_i=\bu_i$ for all $i\in\calV(K)$
and loop to the next cell in $\calT$.  The set $\calV(K)^-$ being
empty should be the most common situation because ``good'' numerical
methods producing the ``unlimited'' states $\{\bu_i\}_{i\in\calV}$
should be such the set $\{K\in\calT\st \calV(K)^-\not=\emptyset\}$ is
sparse.

Since ``\texttt{IF}'' statements are detrimental to vectorization, it
is preferable to define vectorizable masks as follows
\begin{equation}
  \ell_i^+\eqq
  \begin{cases}
    1 & \text{$\Psi^i(\bu_i) > 0$}\\
    0 & \text{otherwise}
  \end{cases}
\quad
   \ell_i^-\eqq
  \begin{cases}
    1 & \text{$\Psi^i(\bu_i) < 0$}\\
    0 & \text{otherwise}
  \end{cases}
  \quad
   \ell_i^0\eqq 1 - \ell_i^+ -\ell_i^- \label{masks}
\end{equation}
and proceed with the algorithm regardless whether the sets $\calV(K)^-$ and
$\calV(K)^+$ are not empty.\medskip

\step{cell.lim.local.mass.step2} The main idea of the algorithm
consists of taking from the sates in $V(K)^+$ to give to the states in
$V(K)^-$ with a balanced budget, pretty much as in the same spirit as
the ``repair'' algorithm from
\cite{Kucharik_Shashkov_Wendroff_JCP_2003,Shashkov_Wendroff_JCP_2004}
or the sweeping method from
\cite{Liu_Cheng_Shu_JSC_2017,Griffin_Shu_RMathSci_2026}, one key
  difference being that here we balance $\Real^q$-valued objects
  whereas only scalar-valued objects are balanced therein.  The
  balancing is done by pushing those states that are out of bounds
  towards some ``good'' states that are in bounds and pushing the
  states that are in bounds towards some sates that are out of bonds
  while maintaining local conservation.

There are many ways to define the ``good'' states.  One possibility
simply consists of using a convex combination of sates in
$\calV(K)^+$, say
\begin{multline}
  \bu_i^+ \eqq \sum_{j\in\calV^+(K)} \theta_{j,K} \bu_j \quad \forall
  i\in\calV(K)^+, \\ \text{with} \sum_{j\in\calV^+(K)}
  \theta_{j,K}=1\quad\text{and}\quad \theta_{j,K}\in[0,1]\ \forall
  j\in\calV(K)^+,
  \label{def_of_ujplus}
\end{multline}
where the convexity coefficients $\theta_{j,K}$ are
user-dependent. One could for instance use $\theta_{j,K}\eqq
m_{j,K}/\sum_{j\in\calV(K)^+} m_{j,K}$ or $\theta_{j,K}\eqq
1/\sum_{j\in\calV(K)^+} 1$. Although convexity implies that
$\bu_i^+\in\calB_l$, \ie $\Psi(\bu_i^+)>0$, the choice
\eqref{def_of_ujplus} may be sub-optimal since it does not guaranteed
that $\Psi^j(\bu_i^+)\ge 0$ for all $j\in\calV^{-}$ (recall that $l$
is the index of the limiting functional $\Psi$; see \eqref{Ass_Psi}).

A better possibility consists of using states $\bu_i^+$ that are a
priori known to satisfy the bounds $\Psi^j(\bu_i^+)\ge 0$ for all
$j\in\calV^{-}$ and all $i\in\calV$. In the context of nonlinear
conservation equations one can often construct first-order accurate
solutions $\{\bu_i^+\}_{i\in\calV}$ that are invariant-domain
preserving, \ie $\Psi(\bu_i^+)>0$ for all $i\in\calV$ (and
``non-oscillatory'' in some heuristic sense). Then adopting the
definition \eqref{def_of_Psi_i} for the local functionals $\Psi^j$ and
setting
\begin{equation}
  \Psi^{i,\min}\eqq \min_{j\in\calV(i)}\Psi(\bu_j^+),\quad\text{with}\quad
  \calV(i)\eqq \bigcup_{K'\in\calT(i)}\calV(K'),
\end{equation}
we indeed have $\Psi^j(\bu_i^+)\ge 0$ for all $j\in\calV(K)^{-}$
because $\calV(K)^-\subset \calV(i)$.

We henceforth assume that the ``good'' states $\bu_i^+$ are defined
such that
\begin{equation}
 \Psi^j(\bu_i^+)\ge 0\quad \forall j\in\calV(K)^{-},\label{def_good_ujplus}
  \end{equation}
and we will show later how this can be achieved in particular
cases.\medskip

\step{cell.lim.local.mass.step3} We now estimate the budget that must be
borrowed to limit (or ``repair'') the states in $V(K)^-$.
For all $j\in \calV(K)^-$ we consider the function
\begin{equation}\psi^{j}: [0,1]\ni t \mapsto \Psi^j(\bu_j +
  t(\bu_j^+-\bu_j))\in \Real.
\end{equation}
Since by assumption $\bu_j\in \calB_{l-1}$,
$\bu_j^+\in\calB_{l}\subset\calB_{l-1}$, and $\calB_{l-1}$ is convex,
we have $ \bu_j + t(\bu_j^+-\bu_j)=(1-t)\bu_j + t \bu_j^+\in
\calB_{l-1}$ for all $t\in[0,1]$; hence $\psi^j$ is well defined over
the interval $[0,1]$. Notice that $\psi^j(0)<0$ since $j\in
\calV(K)^-$. Note also that since $\psi^j(1)\ge 0$, the set
$\{t\in(0,1]\st \psi^j(t)=0\}$ is not empty. Then we compute the
  smallest number in the set $\{t\in[0,1]\st \psi^j(t)=0\}$ and call
  this number $\lambda_{j}^-$.  Because $\psi^j$ is continuous, this
  set is closed and has therefore a minimum element.  The set
  $\{t\in(0,1]\st \psi^j(t)=0\}$ reduces to a single point if $\psi^j$
    is concave, and the set is a connected interval when $\psi^j$ is
    just quasi-concave (\ie not concave).  To summarize, we set
\begin{equation}
  \lambda_{j}^- \eqq
  \min \{t\in (0,1] \st  \psi^j(t)=0\} \quad
\forall j\in \calV(K)^-.
\end{equation}
This definition implies that $\bu_j+t(\bu_j^+-\bu_j)\in \calB_l$ for
all $t\in[\lambda_{j}^-,1]$, \ie the local bound
$\Psi^j(\bu_j+t(\bu_j^+-\bu_j))\ge 0$ holds for all
$t\in[\lambda_{j}^-,1]$. In particular this implies that
\begin{equation}
  \Psi^j(\bu_j+\lambda_j^-(\bu_j^+-\bu_j))=0.\label{full_limiting_in_VKminus}
\end{equation}

\begin{remark}
  Observe in passing that if we used definition \eqref{def_of_ujplus}
  for $\bu_j^+$ instead of assuming that the states $\bu_j^+$ satisfy
  \eqref{def_good_ujplus}, then the above construction would only
  ensure that $\Psi(\bu_j+t(\bu_j^+-\bu_j))\ge 0$ instead of
  $\Psi^j(\bu_j+t(\bu_j^+-\bu_j))\ge 0$ for all
  $t\in[0,\lambda_{j}^-]$. This choices entails that borrowing from the
  states that are in bounds with respect to the local functional
  $\Psi^j$ gives limited states that are in bounds only with respect
  to the global functional $\Psi$ instead of the local functional
  $\Psi^j$, which is may be sub-optimal.
\end{remark}\medskip

\step{cell.lim.local.mass.step4} We now estimate the budget that is
available. As our objective is to construct an algorithm that is fast
and we therefore want to avoid solving a nonlinear optimization
problem on every cell, we
propose to move all the states that are in bounds in one single
direction (as in \cite{Zhang_Shu_JCP_2010,Zhang_Shu_ProcA_2011}) which
we define to be the average over $K$ of the direction the states out
of bounds have been moved,
\begin{equation}
  \bDelta_K^- \eqq \frac{\sum_{j\in\calV^-(K)} m_{j,K} \lambda_j^-(\bu_j-\bu_j^+)}{\sum_{j\in\calV^-(K)} m_{j,K}\lambda_j^-}.
  \label{def_of_Ukplus_Ukminus}
  \end{equation}
Recall that $\bDelta_K^-\in\Real^q$. For all $j\in \calV(K)^+$ we
consider the function
\begin{equation}
\psi^j: [0,+\infty)\ni \lambda \mapsto \Psi^j(\bu_j + \lambda \bDelta_K^-)\in \Real.
\end{equation}
Note that $\psi^j(0)$ is well defined over the interval $\{\lambda\in
[0,\infty)\st \Psi^j(\bu_j + \lambda \bDelta_K^-)\ge 0\}$ owing to the
  continuity and quasi-concavity of $\Psi^j$. Moreover, the set
  $\{\lambda\in [0,\infty)\st \psi^j(\lambda)=0\}$ has a largest element. We
    denote this element by $\lambda_{j,K}^+$. To summarize, we set
\begin{equation}
\lambda_{j,K}^+ \eqq
  \max\{\lambda\ge 0 \st  \psi^j(\lambda)=0\}\qquad \forall j\in \calV(K)^+, \label{def_lambda_star_plus}
\end{equation}
Notice that the quasi-concavity of $\Psi^j$ implies that
$\Psi^j(\bu_j+\lambda\bDelta_K^-)> 0$ for all $\lambda\in[0,\lambda_{j,K}^+]$.  This
property means that borrowing from the states that are in bounds keeps
them in bounds with respect to their local functional $\Psi^j$, even
if full limiting is not achieved \ie if $0\le \lambda<\lambda_{j,K}^+$. In summary we have
\begin{equation}
  \Psi^i(\bu_j + t\lambda_{j,K}^+ \bDelta_K^-)\ge 0 \quad \forall t\in[0,1].
  \label{full_limiting_for_VKplus}
  \end{equation}

\step{cell.lim.local.mass.step5} Now we balance the budget. We
define scaling parameters $\alpha_K$ and $\beta_K$ in the interval
$[0,1]$ for this purpose,
\begin{align}
  \alpha_K \eqq \min\Big(1,\tfrac{\sum_{j\in\calV(K)^+}
    m_{j,K}\lambda_{j,K}^+}{\sum_{j\in\calV(K)^-} m_{j,K}}\Big),\label{def_of_alpha_K}\\
   \beta_K \eqq \min\Big(1,\tfrac{\sum_{j\in\calV(K)^-} m_{j,K}}{\sum_{j\in\calV(K)^+}
    m_{j,K}\lambda_{j,K}^+}\Big).\label{def_of_beta_K}
\end{align}
The parameter $\alpha_K$ will be used to scale back the borrowing
budget if the available budget is not sufficient, and the parameter
$\beta_K$ will be used to scale back the available budget if it is
larger than the borrowing budget. These definitions imply the
following budget balancing identity, which will be critical to
establish local conservation, holds:
\begin{equation}
  -\alpha_K \sum_{j\in\calV(K)^-} m_{j,K} + \beta_K \sum_{j\in\calV(K)^+} m_{j,K}\lambda_{j,K}^+ =0.
  \label{balance_indentity}
\end{equation}

\begin{remark}[Bound on $\lambda_{j,K}^+$] Notice that the definition 
  \eqref{def_lambda_star_plus} does not give an a priori upper bound
  on $\lambda_{j,K}^+$. But the definition of $\alpha_K$, which gives
  an upper bound on the amount that must be borrowed to achieve full
  limiting, tells us that a sufficient condition for full limiting to
  be achievable is
  \begin{equation}
    \lambda_{j,K}^+  \eqq
    \max\{t\ge [0,\lambda^+_K] \st  \psi^j(t)\ge 0\}\quad \text{with}\quad
      \lambda^+_K\eqq \tfrac{\sum_{j\in\calV(K)} m_{j,K}}{\min_{j\in\calV(K)}
    m_{j,K}}.
    \end{equation}
\end{remark}
  
\step{cell.lim.local.mass.step5} The states are limited and
provisionally stored for each cell by defining
\begin{equation}
 \tbu_{j,K} \eqq
  \begin{cases}
    \bu_j + \alpha_K \lambda_j^-(\bu_j^+ - \bu_j) & \text{if $j\in\calV(K)^-$}\\
    \bu_j + \beta_K \lambda_{j,K}^+\bDelta_K^- & \text{if $j\in\calV(K)^+$}\\
    \bu_j  & \text{if $j\in\calV(K)^0$}.
    \end{cases}
\end{equation}
Using the masking technique introduced in \eqref{masks} together with
the identity $\ell_j^0+\ell_j^++\ell_j^+=1$, this gives
\begin{equation}
 \tbu_{j,K} \eqq \bu_j + \ell_j^-\alpha_K \lambda_j^-(\bu_j^+ - \bu_j)
 +\ell_j^+ \beta_K \lambda_{j,K}^+(\bu_j^- - \bu_j).\label{def_tujk}
\end{equation}
Notice that the unlimited numerical sates are not updated here. One
must wait for all the cells in $\calT$ to be visited to do the actual
limiting as explained in the next section.

\begin{remark}[Division by $0$] To avoid divisions by zero
  when either $V(K)^-$ or $V(K)^+$ is empty, one should proceed as
  follows to properly define $\alpha_K$ and $\beta_K$. Let $m_\flat =
  \min_{K\in\calT}\min_{j\in\calV(K)} m_{j,K}$. Let $\epsilon$ be an
  arbitrary number in $(0,1)$ (say $\epsilon=\frac12$). Set
\begin{align}
  \alpha_K^\epsilon \eqq \min\Big(1,\tfrac{\sum_{j\in\calV(K)^+}
    m_{j,K}\lambda_{j,K}^+}{\max(\sum_{j\in\calV(K)^-} m_{j,K}, \epsilon m_\flat)}\Big),
  \label{def_of_alpha_K_epsilon}\\
   \beta_K^\epsilon \eqq \min\Big(1,\tfrac{\sum_{j\in\calV(K)^-} m_{j,K}}{\max(\sum_{j\in\calV(K)^+}
    m_{j,K}\lambda_{j,K}^+, \epsilon m_\flat)}\Big),\label{def_of_beta_K_epsilon}
\end{align}
and replace $\alpha_K$ and $\beta_K$ by $\alpha_K^\epsilon$ and
$\beta_K^\epsilon$ in \eqref{def_tujk}. Let us denote
$\tbu_{j,K}^\epsilon$ the new object given by \eqref{def_tujk}. If $V(K)^+\cup V(K)^-\ne
\emptyset$, then $\alpha_K^\epsilon= \alpha_K$ and $\beta_K^\epsilon=
\beta_K$ and nothing is changed, \ie $\tbu_{j,K}^\epsilon=\tbu_{j,K}$.
If $V(K)^+= \emptyset$, then $\alpha_K^\epsilon= 0$ and $\ell_j^+=0$,
and we obtain $\tbu_{j,K} \eqq \bu_j$ as we should.  If $V(K)^-=
\emptyset$, then $\beta_K^\epsilon= 0$ and $\ell_j^-=0$, and we obtain
$\tbu_{j,K} \eqq \bu_j$ as we should.  Hence, using the definition
\eqref{def_tujk} with the masks and the definitions
\eqref{def_of_alpha_K_epsilon}--\eqref{def_of_beta_K_epsilon} gives the
expected values of $\tbu_{j,K}$ for all $\epsilon \in (0,1)$. Note also that
the balance identity \eqref{balance_indentity} also holds true for all
$\epsilon \in (0,1)$.
  \end{remark}

\medskip

\step{cell.lim.local.mass.step6} Once the cells in $\calT$ have all
been visited, the limiting is done by performing averages over the
vertices.  More precisely, the action of the limiting operator $\calL$
is defined by
\begin{align}\label{cell_limiting_with_local_mass}
  \calL(\bu)_i &=
    \frac{\sum_{K \in\calT(i)}  m_{i,K}\tbu_{i,K}}{\sum_{K \in\calT(i)}  m_{i,K}}\quad \forall i\in\calV.
\end{align}

\subsection{Properties of the Jacobi algorithm}

We now summarize the properties of the
algorithm defined in Steps~\ref{cell.lim.local.mass.step1}--\ref{cell.lim.local.mass.step6}.
with the definition \eqref{def_of_alpha_K_epsilon}-\eqref{def_of_beta_K_epsilon}.
\begin{lemma}[Local limiting]
  \label{Lem:cell.local_general_limiting}
  The limiting mapping $\calL:(\calB_{l-1})^I \to \Real^{I}$ defined
  in
  Steps~\ref{cell.lim.local.mass.step1}--\ref{cell.lim.local.mass.step6}
  has the following properties:

  \textup{(i)} It is conservative: $\sum_{i\in\calV}
  m_i (\calL(\bu))_i = \sum_{i\in\calV} m_i \bu_i$.

  \textup{(ii)}
  Let $i$ in $\calV$ be such that $\Psi^i(\bu_i)\ge 0$. Then
  $\Psi^i(\calL(u)_i))\ge 0$; meaning that if $\bu_i$ is in bounds,
  then $\calL(u)_i$ stays in bounds.

   \textup{(iii)} If all the states are in bounds, then $\calL(\bu)=\bu$;
   \ie the action of $\calL$ is an involution if all the states are in bounds. 

   \textup{(iv)}
 Let $i\in \calV$ be such that $\Psi^i(\bu_i)<0$.
  Assume that the following holds true for $K\in\calT(i)$:
  $\tfrac{\sum_{j\in\calV(K)^+}m_{j,K}
    \lambda_{j,K}^+}{\sum_{j\in\calV(K)^-}m_{j,K}}\ge 1$.
  Then the limited state $\calL(\bu)_i$ is in bounds, \ie $\psi_l^i(\calL(\bu)_i)\ge 0$.
\end{lemma}
\begin{proof}
  \textup{(i)} We start by estimating the local conservation
  defect on each cell for the provisional limited states.
  Let $K\in \calT$ and let 
  \[
  \bdelta_K\eqq \sum_{i\in\calV(K)} m_{i,K}
  (\tbu_{i,K} - \bu_i).
  \]
  Then using the definition of $\bDelta_K^-$ and the balance identity
  \eqref{balance_indentity}, we obtain
  \begin{align*}
    \bdelta_K ={} &\!\sum_{j\in\calV(K)^0}\!  m_{j,K}\bu_{j,K}
     +   \!\sum_{j\in\calV(K)^+}\!  m_{j,K}(\bu_j + \beta^\epsilon_K \lambda_{j,K}^+\bDelta_K^-) \\
    &+ \!\!\sum_{j\in\calV(K)^-}\! m_{j,K}(\bu_j+ \alpha^\epsilon_K\lambda_j^-(\bu_j^+ - \bu_j))
    - \sum_{j\in\calV(K)} m_{j,K}\bu_j\\
     ={}&  \beta^\epsilon_K \bDelta_K^- \!\sum_{j\in\calV(K)^+}\! m_{j,K}\lambda_{j,K}^+  +
 \alpha^\epsilon_K \!\sum_{j\in\calV(K)^-}\! m_{j,K}\lambda_j^-(\bu_j^+ - \bu_j)\\
 ={}& \bDelta_K^- \Big(\beta^\epsilon_K \!\sum_{j\in\calV(K)^+}\! m_{j,K}\lambda_{j,K}^+
 - \alpha^\epsilon_K \!\sum_{j\in\calV(K)^-}\! m_{j,K}\Big)
    =\bzero.
  \end{align*}
  This argument shows that $\sum_{i\in\calV(K)} m_{i,K}(\tbu_{i,K} -
  \bu_i)=\bzero$. Then, the total conservation defect $\bdelta$
  is computed as follows:
  \begin{align*}
   \bdelta\eqq \sum_{i\in\calV} m_i (\calL(\bu)_i-\bu_i)
    & = \sum_{i\in\calV} \frac{m_i}{\sum_{K\in\calT(i)} m_{i,K}}
    \sum_{K\in\calT(i)} m_{i,K} (\tbu_{i,K}-\bu_i).
  \end{align*}
  Using \eqref{def_mik}, \ie
  $m_i = \sum_{K\in\calT(i)} m_{i,K}$, and \eqref{Fubini} we infer that
\begin{align*}
    \bdelta
    & = \sum_{i\in\calV} \sum_{K\in\calT(i)} m_{i,K} (\tbu_{i,K}-\bu_i)
    =\sum_{K\in\calT} \sum_{i\in\calV(K)} m_{i,K} (\tbu_{i,K}-\bu_i)
    =\bzero.
  \end{align*}
  The assertion is proved.

 \textup{(ii)} Let $j\in\calV$ be such that $\Psi^j(\bu_j)> 0$.  Then
 $j\in\calV(K)^+$ for all $K\in\calT(j)$ and $\tbu_{j,K}= \bu_j +
 \beta^\epsilon_K \lambda_{j,K}^+\bDelta_K^-$.  As $\Psi^j(\bu_j)\ge 0$,
 $\Psi^j(\bu_j + \lambda_{j,K}^+\bDelta_K^-)\ge 0$, and $\Psi^j$ is
 quasi-concave, we have $\Psi^j((\bu_j + t\bDelta_K^-)\ge 0$ for all
 $t\in [0,\lambda_{j,K}^+]$.  Hence $\Psi^j(\tbu_{j,K})\ge 0$ because
 $\beta^\epsilon_K\lambda_{j,K}^+\in [0, \lambda_{j,K}^+]$.  Since $\Psi^j$ is
 quasi-concave and the
 definition~\eqref{cell_limiting_with_local_mass} implies that
 $\calL(u)_j$ is in the convex hull of
 $\{\tbu_{j,K}\}_{K\in\calT(j)}$, we conclude that
  \[
  \Psi^j(\calL(\bu)_j)\ge \min_{K\in\calT(j)} \Psi^j(\tbu_{j,K})\ge 0. 
  \]
  The proof is even simpler if $\Psi^j(\bu_j)= 0$ since in this case
 $j\in\calV(K)^0$ for all $K\in\calT(j)$ and $\tbu_{j,K}=
  \bu_j$.
  
  \textup{(iii)} If all the states in $\calV$ are in bounds, then
  $\Psi^i(\bu_i)\ge 0$ for all $i\in\calV$ and $\calV(K)^-$ is the
  empty set for all $K\in\calT(i)$. Then \eqref{def_of_beta_K} implies
  that $\beta^\epsilon_K=1$ and we conclude that $\tbu_{i,K} = \bu_i$ for all
  $i\in\calV$ and all $K\in\calT(i)$.  Then
  \eqref{cell_limiting_with_local_mass} implies that $\calL(\bu)_i =
  \bu_i$.
  
  \textup{(iv)} Let $i\in \calV$ be such that $\Psi^i(\bu_i)<0$.
  Assume that the following holds true for $K\in\calT(i)$:
  $\tfrac{\sum_{j\in\calV(K)^+}m_{j,K}
    \lambda_{j,K}^+}{\sum_{j\in\calV(K)^-}m_{j,K}}\ge 1$. This essentially
  means that there is enough budget available to limit all the states
  in $V(K)^-$. More precisely, this implies that $\alpha^\epsilon_K=1$ and
  $\tbu_{j,K} = \bu_j + \lambda_j^-(\bu_j^+-\bu_j)$. Hence,
  $\Psi^j(\tbu_{j,K})\ge 0$ for all $j\in \calV(K^-)$ owing to
  \eqref{full_limiting_in_VKminus}. Moreover, using
  \eqref{full_limiting_for_VKplus} and $\beta^\epsilon_K\le 1$, we know that
  the states $\tbu_{j,K} = \bu_j + \beta^\epsilon_K\lambda_{j,K}^+\Delta_K^-$
  remain in bounds. In conclusion we have $\Psi^j(\tbu_{j,K})\ge 0$
  for all $j\in \calV(K)$.  The quasi-concavity of $\Psi_l^i$ and the
  definition \eqref{cell_limiting_with_local_mass} then imply that
  $\Psi^i(\calL(\bu)_i)\ge 0$.
\end{proof}

\section{Global limiting} \label{Sec:Global_limiting}

As the local limiting algorithm described in \S\ref{Sec:Jacobi} is not
guaranteed to converge in a fixed number of iterations independent of
the cardinality of the index set $\calV$, we now propose a global
post-processing technique $\calL\upg:(\calB_{l-1})^I\to \calB_l$ (the
super-index $\upg$stands for ``global'') that proceeds as in
\citep[\S2.3]{Guermond_Wang_JCP_2025}. The objective of this algorithm
is only purely theoretical. Its purpose is simply to guarantee that
$\Psi_l((\calL\upg\bu)_i)\ge 0$ for all $i\in\calV$ while preserving
conservation in the rare event that there exists one state that is
still not in bounds with respect to the global function $\Psi_l$ after
a few iterations of the locally conservative redistribution
limiting algorithm. The bulk of the limiting is actually done by the local
algorithms described in \S\ref{Sec:Jacobi}. A similar idea is
advocated in \cite[\S6]{Loubere_Staley_Wendroff_JCP_2006} for
``repair'' methods.

\subsection{The algorithm}\label{sec:global_limiting_algorithm}
As the algorithm consists of one global limiting that post-processes
all the states in the list $\calV$, we assume to have at hand a set of
scalar $(\alpha_j)_{j\in\calV}$ that a priori indicate how much each
state $\bu_j$ is able to give away if it is well inside $\calB_l$ (or
$\calB$ for short).  For instance, when simulating an hyperbolic
system with compactly supported initial data, it is not reasonable to
post-process states that not in the domain of dependence of the
initial data since they should not change until they enter the domain
of dependence of the data.
The algorithm consists of the following steps.\medskip

\newStep
\step{global.step1} We first partition $\calV$ into three sets $\calV^+$,
$\calV^-$, and $\calV^0$ as follows:%
\begin{subequations}%
  \begin{align}%
  \calV^+ &\eqq \{i\in\calV \st \Psi(\bu_i) > 0\},\\
  \calV^- &\eqq \{i\in\calV \st \Psi(\bu_i) < 0\},\\
  \calV^0 &\eqq \{i\in\calV \st \Psi(\bu_i) = 0\}.
  \end{align}\label{def_of_the_Vs_global}%
\end{subequations}
As in \eqref{masks} we define the masks
\begin{equation}
  \ell_i^+\eqq
  \begin{cases}
    1 & \text{$\Psi(\bu_i) > 0$}\\
    0 & \text{otherwise}
  \end{cases}
\quad
   \ell_i^-\eqq
  \begin{cases}
    1 & \text{$\Psi(\bu_i) < 0$}\\
    0 & \text{otherwise}
  \end{cases}
  \quad
   \ell_i^0\eqq 1 - \ell_i^+ -\ell_i^- \label{global_masks}
\end{equation}

\step{global.step2} Arguing as in Step~\ref{cell.lim.local.mass.step3}
of the local limiting algorithm, we define the real number
$\lambda_j^- \in (0,1]$ for all $j\in \calV^-$ so that
\begin{align}
&\lambda_{j}^-\eqq\max \{t\in [0,1] \st  \Psi(\bu_j + t(\bu_j^+-\bu_j))=0\}.
\end{align}

\step{global.step3} Next we
define a direction that is colinear to the average direction the states
out of bounds have been moved.
We will use this direction to move all the
states that are in bounds to define the global budget that is available,
\begin{equation}
  \bDelta^- \eqq
  \frac{\sum_{j\in \calV^-} m_j \lambda_j^-(\bu_j-\bu_j^+)}{\sum_{j\in\calV^-}
    m_j\lambda_j^-}.
\end{equation}
Then we proceed as in Step~\ref{cell.lim.local.mass.step4} of the
local limiting algorithm and define $\lambda_j^+$ by
\begin{equation}
  \lambda_j^+\eqq \max\{\lambda\st \Psi(\bu_j + \lambda\alpha_j \bDelta^-)=0\}
  \quad \forall j\in\calV^+.
\end{equation}
Notice that the quasi-concavity of $\Psi$ implies that
\begin{equation}
  \Psi(\bu_j + t \lambda_j^+\alpha_j \bDelta^-)\ge 0 \quad \forall t\in [0,1].
  \label{lambda_plus_global_property}
  \end{equation} 

\step{global.step4} We now balance the global budget by defining a
global parameter:
\begin{equation}
\beta\eqq \frac{\sum_{j\in \calV^-} m_j \lambda_j^-}{\sum_{j\in \calV^+}
 m_j  \lambda_j^+\alpha_j}.
\end{equation}

\step{global.step5}
We finally define the global limiting operation
$\calL\upg$:
\begin{equation}
  \calL\upg(\bu)_j \eqq
  \begin{cases}
    \bu_j + \lambda_j^-(\bu_j^+ - \bu_j) & \text{if $j\in\calV^-$}\\
    \bu_j + \beta \lambda^+\alpha_j\bDelta & \text{if $j\in\calV+$}\\
    \bu_j  & \text{if $j\in\calV^0$}.
    \end{cases}\label{global_def_global_limiting}
\end{equation}
Using the masks defined in  \eqref{global_masks} this gives
\begin{equation}
  \calL\upg(\bu)_j  = \bu_j+\ell_j^+\beta \lambda_j^+\alpha_j\bDelta
  +\ell^+\lambda_j^-(\bu_j^+ - \bu_j).
\end{equation}

\subsection{Properties of the global limiting algorithm}

We summarize here the properties of the global limiting
algorithm~\ref{global.step1}--\ref{global.step5} defined above.

\begin{lemma}[Global limiting] \label{Lem:cell.global_general_limiting}
  The limiting mapping $\calL\upg: (\calB_{l-1})^I \to \Real^I$
  defined in Steps~\ref{global.step1}--\ref{global.step5} in
  \S\ref{sec:global_limiting_algorithm} has the following properties:

  \textup{(i)} It is conservative: $\sum_{i\in\calV} m_i (\calL\upg(\bu))_i=
  \sum_{i\in\calV} m_i\bu_i$.

  \textup{(ii)} It maps $(\calB_{l-1})^I$ to $(\calB_l)^I$ if
  $\sum_{j\in \calV^-} m_j \lambda_j^-\le \sum_{j\in \calV^+}
 m_j  \lambda_j^+\alpha_j$.

  \textup{(iii)} It is an involution: $\bu\in(\calB_l)^I$ implies that
  $\calL\upg(\bu)=\bu$.

\end{lemma}

\begin{proof}
 \textup{(i)}  The conservation defect is given by 
  \begin{align*}
 \bdelta \eqq{}& \sum_{j\in \calV^0} m_j \bu_j  + \sum_{j\in \calV^-}
 m_j\left(\bu_j + \lambda_j^-(\bu_j^+-\bu_j)\right) \\
  & + \sum_{j\in\calV^+} m_j (\bu_j +\beta \lambda_j^+ \alpha_j\bDelta^-)
 -\sum_{j\in\calV} m_j \bu_j \\
 ={}&  \sum_{j\in \calV^-} m_j \lambda_j^-(\bu_j^+-\bu_j)
 + \sum_{j\in\calV^+} m_j \lambda_j^+ \alpha_j \beta \bDelta^-\\
={}&  \big(\beta \sum_{j\in \calV^+}
 m_j \lambda_j^+ \alpha_j - \sum_{j\in \calV^-} m_j \lambda_j^-\big)\bDelta^-.
  \end{align*}
  We conclude that $\bdelta=\bzero$ by using the definition of
  $\bdelta^-$, \ie the algorithm is conservative.

 \textup{(ii)} If
  $\sum_{j\in \calV^-} m_j \lambda_j^-\le \sum_{j\in \calV^+}
 m_j  \lambda_j^+\alpha_j$, then $\beta\le 1$ and \eqref{lambda_plus_global_property}
 implies that $\Psi(\calL\upg(u)_j)\ge 0$, \ie  $\calL\upg(u)_j\in \calB_l$
 for all $j\in\calV^+$. By definition $\calL\upg(u)_j\in \calB_l$ for all
 $j\in \calV^0\cup \calV^-$. Hence $\calL\upg(u)_j\in \calB_l$ for all
 $j\in\calV$.

 \textup{(iii)} If all the states are in bounds, then $\calV^-$ is the
 empty set and $\beta=0$; hence, the states in $\calV^+$ are
 untouched. By definition, the states in $\calV^0$ are also untouched.
 In conclusion $\calL\upg$ is indeed an involution.
\end{proof}

\begin{remark} Global limiting is effective only if
  \begin{equation}
    \sum_{j\in \calV^-} m_j \lambda_j^-\le \sum_{j\in \calV^+} m_j
    \lambda_j^+\alpha_j.
  \end{equation}
  This inequality is likely to hold if the cardinality of $\calV^-$ is
  significantly smaller than that of $\calV^-$ and the window
  $\alpha_j$ allowed for the states that are in bounds to move is
  large enough. The obvious choice is $\alpha_j=1$, but for
  time-dependent problem we could take
\begin{equation}
  \alpha_j \eqq \frac{|\Psi(\bu_j^{n+1}) - \Psi(\bu_j^{n})|}{\alpha}, \qquad
  \alpha \eqq \frac{\sum_{j\in \calV}m_j |\Psi(\bu_j^{n+1}) - \Psi(\bu_j^{n})|}{
    \sum_{j\in \calV}m_j}.
  \end{equation}
\end{remark}

\section{Examples of discretization settings}
\label{Sec:Examples_of_discretzation_settings}

We illustrate the setting described in \S\ref{Sec:Jacobi} and show
that the structures described in \S\ref{Sec:Enumeration} can be
realized using nodal and modal finite elements, and Fourier expansions
to solve nonlinear conservation equations.

\subsection{Generic finite elements} \label{Sec:generic_degree_FE}
Let $\famTh$ be a shape-regular sequence of meshes of some polygonal
domain $\Dom$ in $\Real^d$.  Let $\calT_h$ be one mesh in this
sequence and let us drop the index $h$ from now on.  We are going to
abuse the notation and identify $\calT$ with the index set enumerating
the cells in $\calT$. We also abusively identify cells with their
indices, \ie we are going use the abuse of notation
$\calT=\bigcup_{K\in\calT} K$. Let $P(\calT)$ be a scalar-valued
finite elements space generated on the mesh $\calT$ (continuous or
discontinuous) and using polynomials of degree $k$.  Let
$\{\varphi_i\}_{i\in\calV}$ be the global shape functions of
$P(\calT)$. Notice that here $\calV$ is the index set enumerating all
the global shape functions. For all $K$ in $\calT$, we define the set
$\calV(K)$ to be the collection of the indices $i$ of those shape
functions $\varphi_i$ that are such $\varphi_{i|K} \not \equiv 0$. The
cardinality of $\calV(K)$ only depends on the polynomial degree of the
approximation $k$, \ie $\card(\calV(K))=\calO(k^d)$. Then, for all
$i\in\calV$, the set $\calT(i)$ is the collection all those cells
$K\in\calT$ so that $\varphi_{i|K}\not \equiv 0$.  The cardinality of
$\calT(i)$ is equal to $1$ if the finite element space is
discontinuous. The cardinality of $\calT(i)$ is bounded from above by
a number that only depends on the shape-regularity of the mesh
sequence if the finite element space $P(\calT)$ is $C^0$-conforming
(\ie the cardinality of $\calT(i)$ does not depend on the polynomial
degree $k$ and the mesh size $h$). To conclude, here we have
  \begin{align}
    \calV(K) &=\{K\in\calT\st \varphi_{i|K} \not\equiv 0\}\quad\forall K\in\calT,\\
      \calT(i) &=\{i\in\calV\st \varphi_{i|K} \not\equiv 0\}\quad \forall i\in\calV.
  \end{align}%
\begin{figure}[h]\vspace{-\baselineskip}%
  \centering \scalebox{0.23}{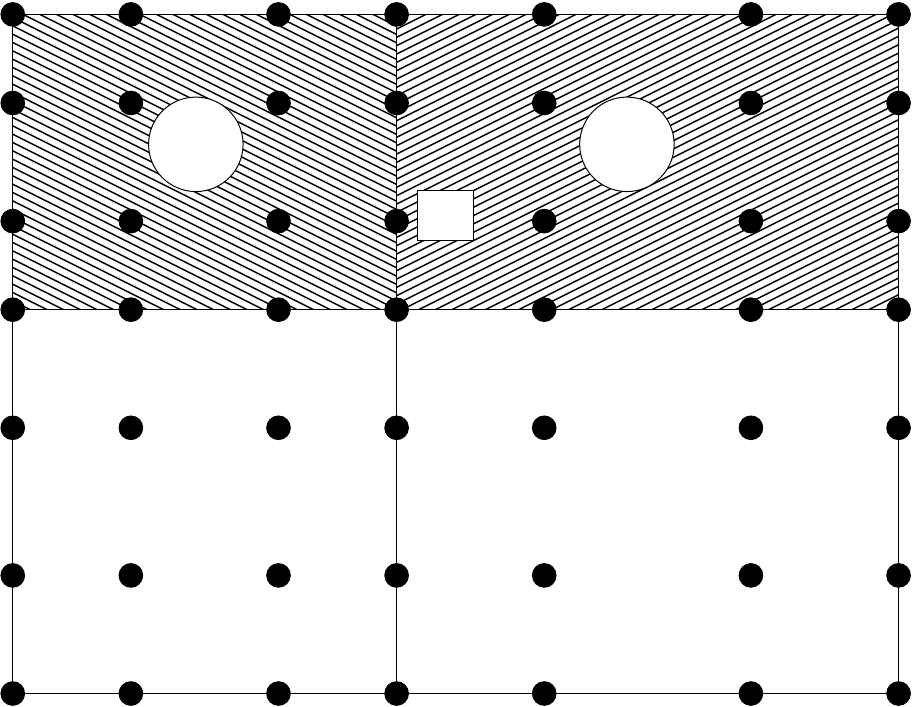}
    \hfil\scalebox{0.18}{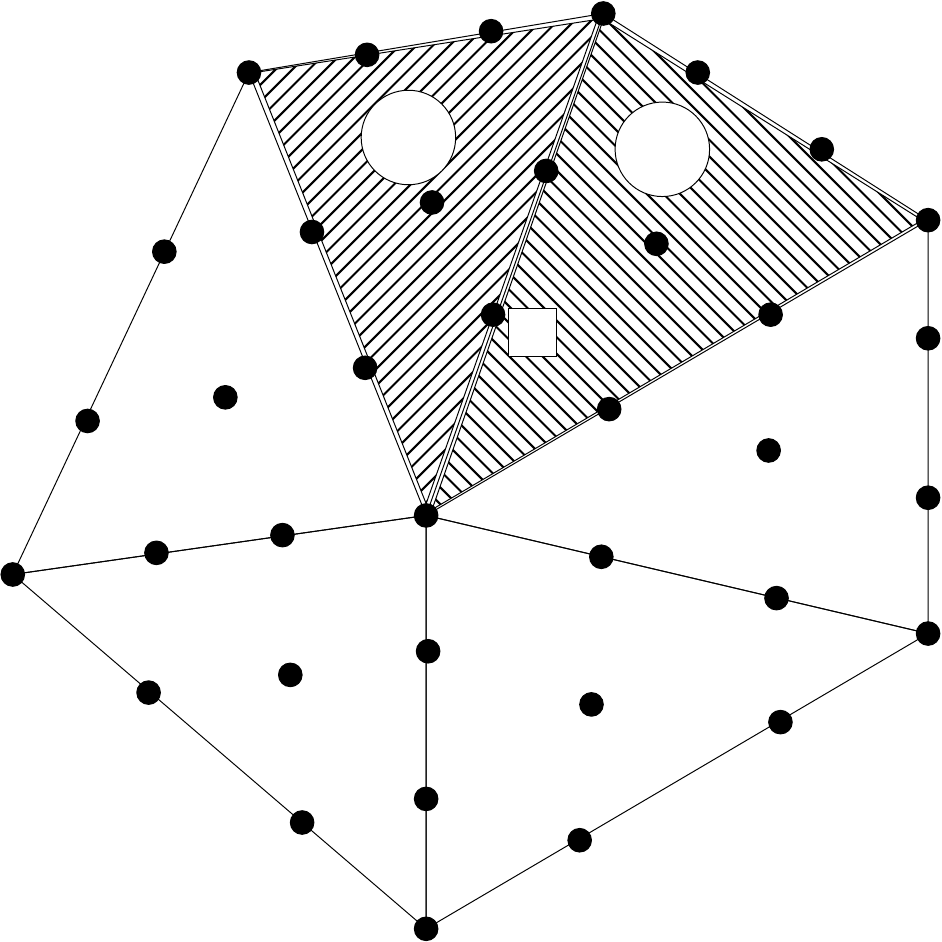}
     \hfil\scalebox{0.18}{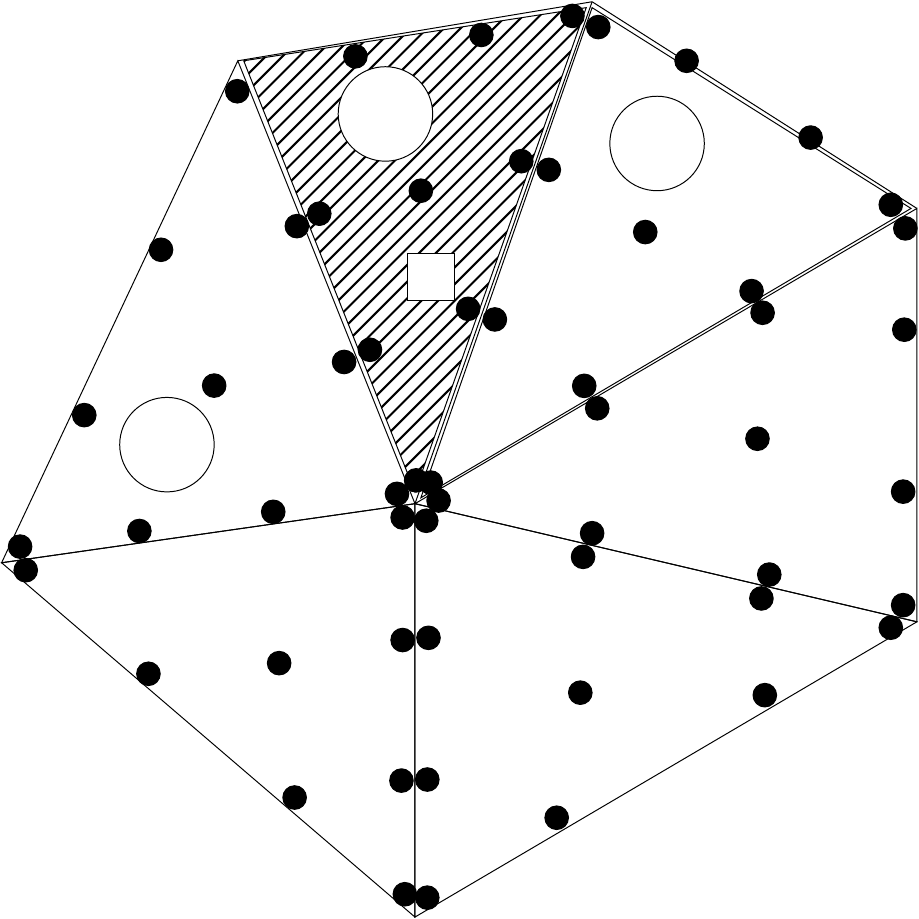}
\caption{Index set $\calT(i)$ (shaded cells) for
  cubic two-dimensional finite elements. Left: Continuous 
  $\polQ_{3,2}$ elements on quadrangles, $\calT(i)=\{K,K'\}$.  Center:
  Continuous $\polP_{3,2}$ elements on triangles,
  $\calT(i)=\{K,K'\}$. Right: Discontinuous $\polP_{3,2}$
  elements on triangles, $\calT(i)=\{K\}$.  Black dots represent
  Lagrange nodes for Lagrange
  polynomials, or domain points for Bernstein
  polynomials, or quadrature points (Gauss-Lobatto) for modal  shape
  functions.}
\label{Fig:Notation_calTi_high}
\end{figure}

Let $\polP_{k,d}$ be the space of the $d$-variate polynomials of
total degree at most $k$ and $\polQ_{k,d}$ be the space of the
$d$-variate polynomials of partial degree at most $k$. We illustrate
in Figure~\ref{Fig:Notation_calTi_high} the definition of $\calT(i)$ for
$\polP_{3,2}$ continuous finite elements in the left panel and
discontinuous finite elements in the right panel.

Let $i\in\calV$. Assuming that either the lumped mass $m_i$ is
positive or $\varphi_i\ge 0$, two possible ways to define $m_i^K$ so
that the generic property \eqref{def_mik} holds are
\begin{equation}
    m_i^K \eqq m_i \frac{\mes{K}}{\sum_{K'\in\calT(i)}\mes{K'}},\quad
    \text{or}\quad
    m_i^K \eqq m_i \frac{\int_K \varphi_i(\bx)\diff
      x}{\int_\Dom \varphi_i(\bx)\diff x}.
\end{equation}

In the context of the approximation of some nonlinear conservation equation with states taking
values  in $\Real^q$, one is then lead to consider the $\Real^q$-valued finite element space
\begin{equation}
  \bP(\calT) = P(\calT)^q.
\end{equation}
The functions in this space have the following representation $\bu_h =
\sum_{i\in\calV} \bsfu_i\varphi_i$.

Assume now that this setting is
used to solve some nonlinear conservation equation with invariant
domain $\calA$. Assuming that $\bu_h$ is the result of some
computation, the key problem is now to make sure that $\bu_h$ is in
bounds. To better formalize this problem we assume that what is at
stake here is to make sure that the states
$(\bu_i\up{Lag})_{i\in\calV}$ are all in $\calA$ where
$(\bu_i\up{Lag})_{i\in\calV}$ are the coefficients of $\bu_h$ in the
Lagrange basis of $\bP(\calT)$.  Hence one must start by making a
change of basis $(\bu_i)_{i\in\calV}\to (\bu_i\up{Lag})_{i\in\calV}$
before applying the conservative redistribution limiting algorithm.

\subsection{Hierarchical decomposition}
\label{Sec:high-degree_FE}
One problem with the finite element setting described in
\S\ref{Sec:generic_degree_FE} is that the index set $\calV(K)$ grows like
$k^d$ where $k$ is the polynomial degree of the approximation and $d$
the space dimension; moreover, localization is somewhat lost when $k$
is very large (think of the limit case where the mesh is composed of
only one cell).  To handle this situation, we
introduce a standard hierarchical decomposition of the space
approximation as done for instance in
\cite{Abgrall_Viville_Beaugendre_Dobrzynski_2017,Abgrall_Bacigaluppi_Tokareva_2019},
\cite{Pazner_2021,Kuzmin_Quezada_2020},
\cite{Guermond_Nazarov_Popov_CMAME_2024},
\cite{Vilar_Abgrall_SISC_2024}.

Instead of using the symbol $\calT$ for the mesh, we now use
$\calT\upH$.  Here, the super-index $\upH$ reminds us that this mesh
is used to construct some high-order approximation with some
polynomial degree $k>1$.  We assume that each cell in $\calT\upH$
can be subdivided as in
\citep{Abgrall_Viville_Beaugendre_Dobrzynski_2017,Abgrall_Bacigaluppi_Tokareva_2019},
\citep{Pazner_2021,Kuzmin_Quezada_2020},
\citep{Guermond_Nazarov_Popov_CMAME_2024},
\citep{Vilar_Abgrall_SISC_2024}. For instance, we assume that for each
cell $K$ in $\calT\upH$, the local shape functions are Lagrange
elements based on Lagrange nodes $\{\bz_{i,K}\}_{i\in\calV\upH(K)}$,
or the Bernstein basis based on domain points
$\{\bz_{i,K}\}_{i\in\calV\upH(K)}$, or a modal basis based on some
quadrature points $\{\bz_{i,K}\}_{i\in\calV\upH(K)}$ (say Fekete
points for simplices or Gauss-Lobatto points for cuboids, see \eg
\cite[Chap.2\&3]{Karniadakis_Sherwin_1999},
\cite{Taylor_Wingate_Vincent_2000},
\cite[Chap.3]{Hesthaven_Warburton_2008}). We then assume that $K$ can
be subdivided by joining the points
$\{\bz_{i,K}\}_{i\in\calV\upH(K)}$.  The subdivision must be such that
each subcell thus created is a legitimate geometric object (\ie
simplex or cuboid) that can be used to construct $\polP_{1,d}$ or
$\polQ_{1,d}$ finite elements.  We call $\calT\upL(K)$ the collection
of the subcells of $K$ thus created. We denote by $\calT\upL$ the
list (and the mesh composed) of all the subcells created by
subdividing all the cells in $\calT\upH$.  We denote by
$P(\calT\upL)$ the scalar valued finite element space based on the
subdivided mesh $\calT\upL$ using piece-wise $\polP_{1,d}$ or
$\polQ_{1,d}$ approximations on $K$, depending whether $K$ is a
simplex or cuboid. The above construction is
illustrated in Figure~\ref{Fig:Notation_calTi}.
  
\begin{figure}[h]
\centering \scalebox{0.23}{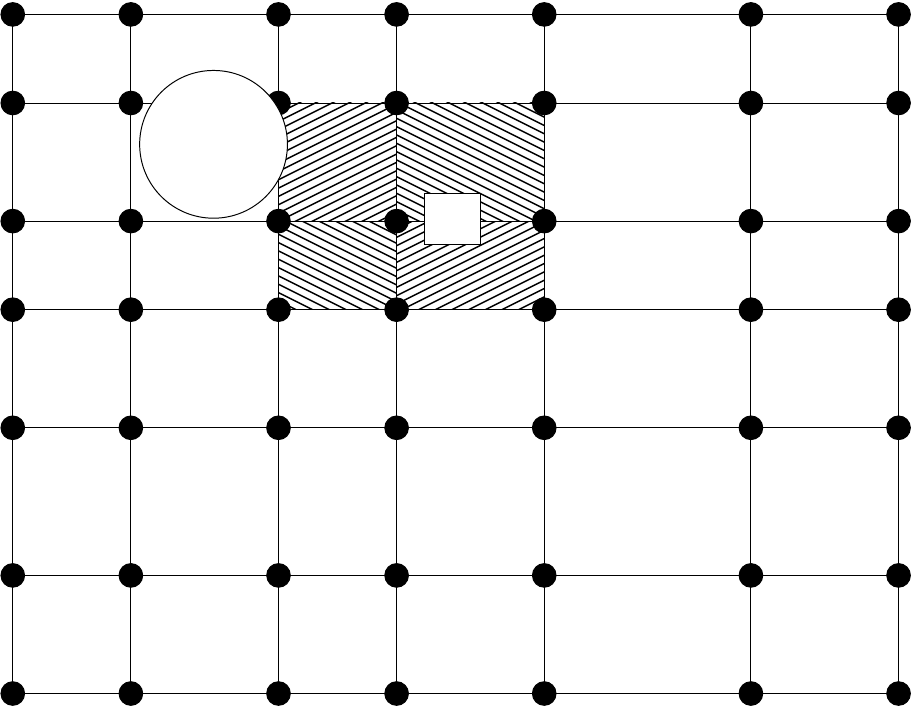}
   \hfil\scalebox{0.18}{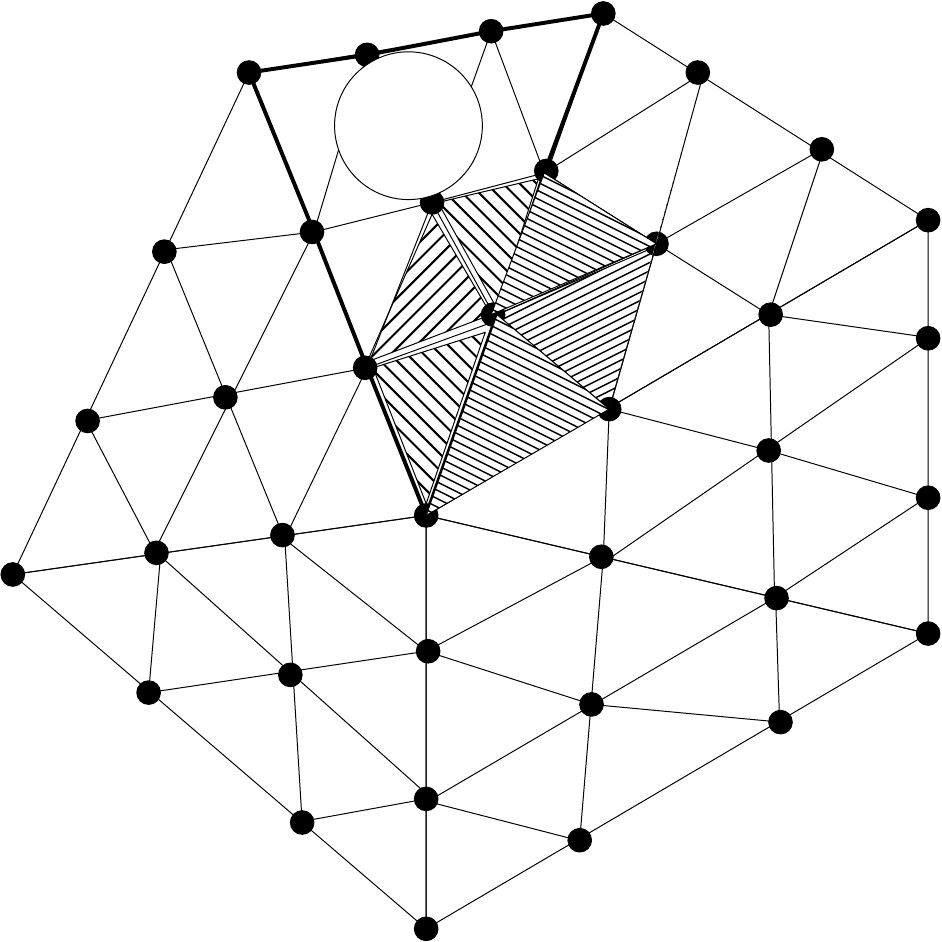}
   \hfil\scalebox{0.18}{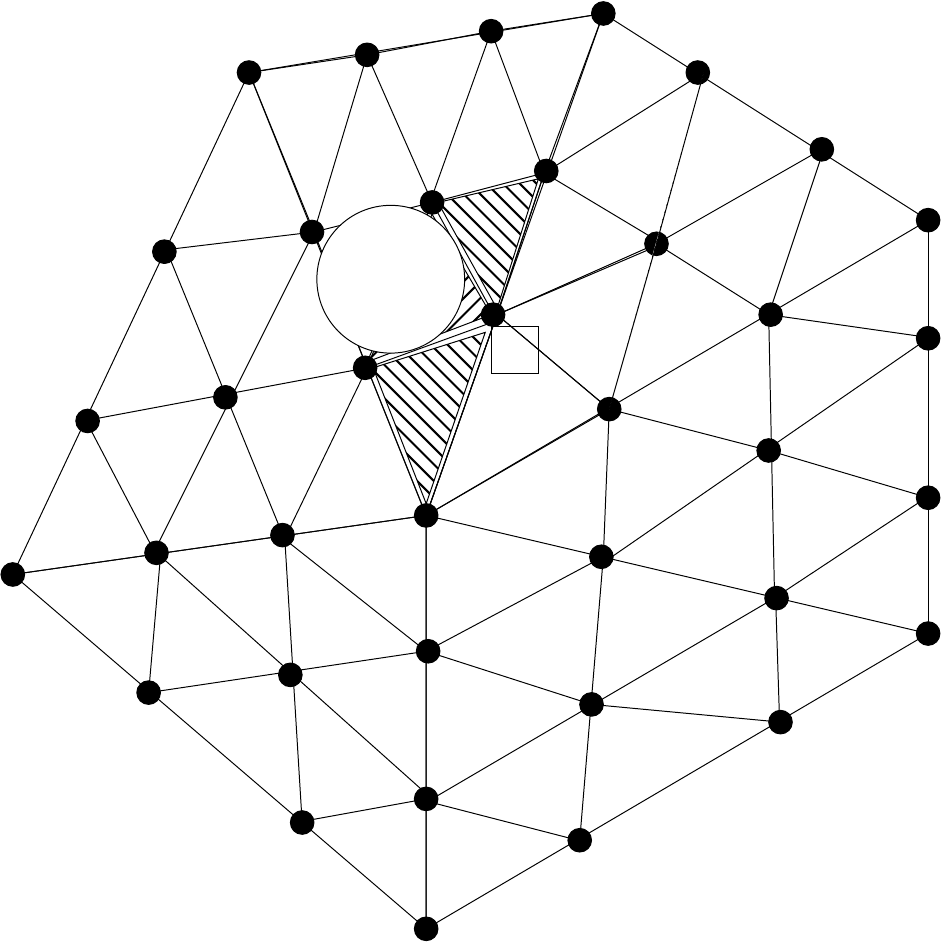}
\caption{Index set $\calT\upL(i)$ (shaded cells) for cubic
  two-dimensional finite elements. Left: Continuous $\polQ_{3,2}$
  finite elements. Center: Continuous $\polP_{3,2}$ 
  elements. Right: Discontinuous $\polP_{3,2}$  elements.}
\label{Fig:Notation_calTi}
\end{figure}

If $P(\calT\upH)$ is composed of continuous elements, we define the
space $P(\calT\upL)$ to be composed of continuous $\polP_1$ or
$\polQ_1$ elements.
Denoting by $\{\bz_{i}\}_{i\in\calV}\eqq \cup_{K\in\calT\up{H},i\in\calV\up{H}(K)}\{\bz_{i,K}\}$,
the global shape functions of $P(\calT\upL)$ are the Lagrange piecewise linear
functions
associated with the nodes $\{\bz_{i}\}_{i\in\calV}$.
Notice that $P(\calT\upH)$ and $P(\calT\upL)$ have
the same dimension.
If $P(\calT\upH)$ is composed of discontinuous elements, we define
$P(\calT\upL)$ to be composed of the functions that are discontinuous across
the interfaces of the macro-cells of $\calT\up{H}$, but  continuous and piecewise
$\polP_1$ or $\polQ_1$ over each macro-cell of $\calT\up{H}$. This
construction guarantees that $P(\calT\upH)$ and $P(\calT\upL)$ have
the same dimension; hence the dofs in each space can be enumerated
with the same index set $\calV$.
In conclusion whether the functions in $P(\calT\upH)$ are continuous
or not, $P(\calT\upL)$ is a Lagrange space with shape functions
associated with the nodes $(\bz_i)_{i\in\calV}$, and the spaces
 $P(\calT\upH)$ and $P(\calT\upL)$ have the same dimension.

We now propose to use the above hierarchical construction to do the
limiting of the states $(\bu_i\up{Lag})_{i\in\calV}$ mentioned at the
end of \S\ref{Sec:generic_degree_FE}.  More precisely we propose to
apply the conservative redistribution limiting from
\S\ref{Sec:Jacobi}--\ref{Sec:Global_limiting} with the pair of index
sets $(\calV,\calT\up{L})$ and the masses $(m_i)_{i\in\calV}$ and
$(m_{i,K})_{i\in\calV, K\in\calT\up{L}(i)}$. The key advantage of this
setting is localization.

\subsection{Fourier approximation}\label{Sec:Fourier}

For Fourier and spectral approximations, the shape functions
$(\varphi_i)_{i\in\calV}$ are global. To formalize the limiting
problem for the global approximation $\bu_h\eqq
\sum_{i\in\calV}\bsfu_i\varphi_i$, we must assume that there exists a
set of (Lagrange) points $(\bz_i)_{i\in\calV}$ where the states
$(\bu_h(\bz_i))_{i\in\calV}$ must be in the invariant set $\calA$.
This usually involve a fast transform.

For the Fourier approximation over the interval $[0,1]$ with the
trigonometric basis $\{\cos(2\pi n
x)\}_{n\in\intset{0}{N}}\cup\{\sin(2\pi n x)\}_{n\in\intset{1}{N}}$,
the points in questions are
$\{\bz_i=\frac{i}{2N+1}\}_{i\in\intset{0}{2N}}$. Setting
$\calV\eqq\intset{0}{2N}$, the fast Fourier transform realizes the
mapping $(\bu_i)_{i\in\calV}\to (\bu_i\up{Lag}\eqq
\bu_h(\bz_i))_{i\in\calV}$. Solving nonlinear PDEs with such an
approximation invariably requires estimating
$(\bef(\bu_i\up{Lag}))_{i\in\calV}$ for some nonlinear function $\bef$
whose domain is some compact interval $\calA$. Assuming that the exact
solution takes values in $\calA$, it is therefore imperative that the
states $(\bu_i\up{Lag})_{i\in\calV}$ all belong to $\calA$. The
conservative redistribution limiting proposed in the paper can be used
for this purpose. It turns out that in this context
\begin{equation}
  \int_0^1 \bu_h(\bx)\diff x = \sum_{i\in\calV} \frac{1}{2N+1}\bu_i\up{Lag}.
  \end{equation}
Hence the corresponding lumped mass coefficients are $m_i\eqq
\frac{1}{2N+1}$ for all $i\in\calV$. The set $\calT$ is simply
composed of $2N+1$ cells $\{K_i\eqq
[\bz_i,\bz_{i+1}]\}_{i\in\intset{0}{2N}}$ with the convention
$\bz_{2N+1}=0$.  This construction gives $\calV(K_i)=\{i,i+1\}$ and
$\calT(i)=\{K_{i-1},K_{i}\}$ with the convention that $K_{-1}\eqq
K_{2N}$. We denote $\bP(\calT\up{L})$ the continuous finite piecewise
linear finite element space based on the mesh $\calT$ (the superscript
$\up{L}$ reminds us that $\bP(\calT\up{L})$ is a low-order
approximation space. We refer the reader to
\cite[\S3]{Liu_Cheng_Shu_JSC_2017} where these ideas are also exposed.

\subsection{Possible definition of $u_j^+$}

One key problem left is the definition of the IDP states $\bu_j^+$ for
all $j\in\calV$ (see Step~\ref{cell.lim.local.mass.step2} of the local
conservative redistribution limiting algorithm). This point heavily
depends on the problem that is at hand.

\subsubsection{Example 1: steady-state linear transport}
\label{sec:definition_of_uK_plus_steady_transport}

Let us specialize the setting and assume that we solve a linear scalar
transport equation like in \citep{Guermond_Wang_JCP_2025} in the
context of the steady-state neutron transport equation. Then for all
$j\in\calV$ the local bounds $\bu_i^+$ can be obtained by using the
method of characteristics over each elementary cell in $\calT\up{L}$ starting
from the Lagrange located at $\bz_i$; the reader is referred to
\citep{Guermond_Wang_JCP_2025} where this computation is done.

\subsection{Example 2: time-dependent hyperbolic systems}
\label{sec:definition_of_uK_plus}

Let us now assume that we are solving a nonlinear time-dependent
hyperbolic system with some invariant domain $\calA$.  Assume that we
are using some explicit time stepping.  Then numerous ways to compute
states $\bu_i^+$ that are guaranteed to be in $\calA$ for all
$i\in\calV$ (up to a CFL condition) can be found in the literature
(too many to be listed here). One technique we know well and that is
indeed guaranteed to be IDP is explained in
\citep{Guermond_Popov_SINUM_2016}. This method is discretization
agnostic and works well with linear finite elements, continuous and
discontinuous; see \citep[4]{Guermond_Popov_Tomas_CMAME_2019} for the
implementation details involving continuous and discontinuous
elements. This method works particularly well with the piecewise
linear space $\bP(\calT\up{L})$ defined above, whether the
approximation space is a high-order finite element space as in
\S\ref{Sec:high-degree_FE} or the Fourier space as in
\S\ref{Sec:Fourier}.

\section{Numerical illustrations} \label{Sec:numerical illustrations}

We now go briefly over a standard series of benchmark problems to
illustrate the versatility of the proposed conservative redistribution
limiting technique. We have verified on tests not reported here for
brevity that the proposed limiting technique is indeed consistent, and
high-order accuracy is maintained when the exact solution is smooth.

\subsection{Linear transport with continuous finite elements}
\label{Sec:three_body_scalar_transport}

We consider the time-dependent linear transport equation with the
velocity $2\pi(-x_2, x_1)\tr$ in the unit disk $\Dom$ centered
at $\bzero$. The initial conditions are as follows:
\begin{equation} \label{Three_Body_2D}
u_0(\bx) \eqq
\begin{cases}
1&  \mbox{if } r_1(\bx) \le r_0 \text{ and
($|x_1| \ge \frac{1}{20}$ or  $x_2 \ge \frac{7}{10}$)},\\
1- \frac{r_2(\bx)}{r_0}  &  \text{if } r_2(\bx)  \le r_0,\\
\frac14\left(1+\cos\big(\frac{r_3(\bx)}{r_0}\pi\big)\right)&  
\mbox{if } r_3(\bx)  \le r_0,\\
0&  \mbox{otherwise, }
\end{cases}
\end{equation}
with $r_0 \eqq 0.3$, $r_1(\bx)\eqq \|\bx-\bx_1\|_{\ell^2}$,
$r_2(\bx)\eqq \|\bx-\bx_2\|_{\ell^2}$, $r_3(\bx)\eqq
\|\bx-\bx_3\|_{\ell^2}$, with $\bx_1\eqq (0,\frac12)$, $\bx_2\eqq
(-\frac12,0)$, $\bx_3\eqq (0,-\frac12)$.

 \begin{table}[ht]\small \centering
  \begin{tabular}{c|c|c}
    \multicolumn{3}{c}{$\polP_1$}\\\hline
I &  $\delta_1$ & rate \\ \hline
   453 & 3.92E-01  &  -- \\
   1716 & 2.39E-01  & 0.75\\
   6497 & 1.35E-01  & 0.85\\
  25049 & 8.04E-02  & 0.77\\
  98752 & 4.72E-02  & 0.78
  \end{tabular}
  \hfil
  \begin{tabular}{c|c|c}
    \multicolumn{3}{c}{$\polP_2$}\\\hline
    I &  $\delta_1$ & rate \\ \hline
   458 & 3.43E-01  &  -- \\
   1746 & 1.95E-01  & 0.84\\
   6735 & 1.12E-01  & 0.82\\
  25734 & 6.75E-02  & 0.75\\
  99690 & 3.98E-02  & 0.78
  \end{tabular}
  \hfil
  \begin{tabular}{c|c|c}
    \multicolumn{3}{c}{$\polP_3$}\\\hline
    I &  $\delta_1$ & rate \\ \hline
     460 & 4.50E-01  &  -- \\
   1711 & 2.44E-01  & 0.93\\
   6751 & 1.29E-01  & 0.93\\
  27091 & 7.54E-02  & 0.78\\
  99859 & 4.59E-02  & 0.76  
  \end{tabular}
  \caption{Convergence tests for Three body problem. From left to right
    $\mathbb{P}_1$, $\mathbb{P}_2$, $\mathbb{P}_3$. $\text{CFL}=0.25$.}
  \label{Table:Three_body_pb}
 \end{table}

 The approximation is done using continuous $\polP_k$,
 $k\in\{1,2,3\}$, finite elements with SUPG stabilization.  The SUPG
 stabilization coefficient is defined to be equal to
 $0.04/((k+1)(k+d))^{\frac12}$ with $d=2$. The meshes are
 quasi-uniform and composed of $\polP_2$ triangles to match the
 boundary of the domain with third-order accuracy.  We use the ERK3
 scheme with maximal efficiency to step in time (\ie equi-distributed
 time stages; not to be confused with the ERK3 SSP method); see
 \cite[Eq. (4.1)]{ErnGu:21+}.  The CFL number for each ERK3 stage is
 0.25. The locally conservative redistribution limiting algorithm is
 applied at the end of each Runge-kutta stage to enforce the local
 maximum and minimum principle.  The local IDP states $\bu_j^+$ in
 each case are estimated using the low-order update with the IDP
 method described in \citep{Guermond_Popov_SINUM_2016} and the
 hierarchical $\polP_1$ meshes $\calT\up{L}$ described in
 \S\ref{Sec:high-degree_FE}.  As the exact solution is only BV, the
 maximal theoretical rate of convergence in space is $\calO(h^1)$ in
 the $L^1$-norm. We report in Table \eqref{Table:Three_body_pb} the
 observed convergence rates with the $L^1$-norm of the error computed
 at $t=1$.

We show in Figure~\ref{Fig:three_body} a three-dimensional rendering
of the solution after one revolution using $\polP_1$, $\polP_2$, and
$\polP_3$ continuous finite elements with, respectively, $25\,049$,
$25\,734$ and $27\,091$ dofs.

\begin{figure}[htb]
\includegraphics[width=0.32\textwidth]{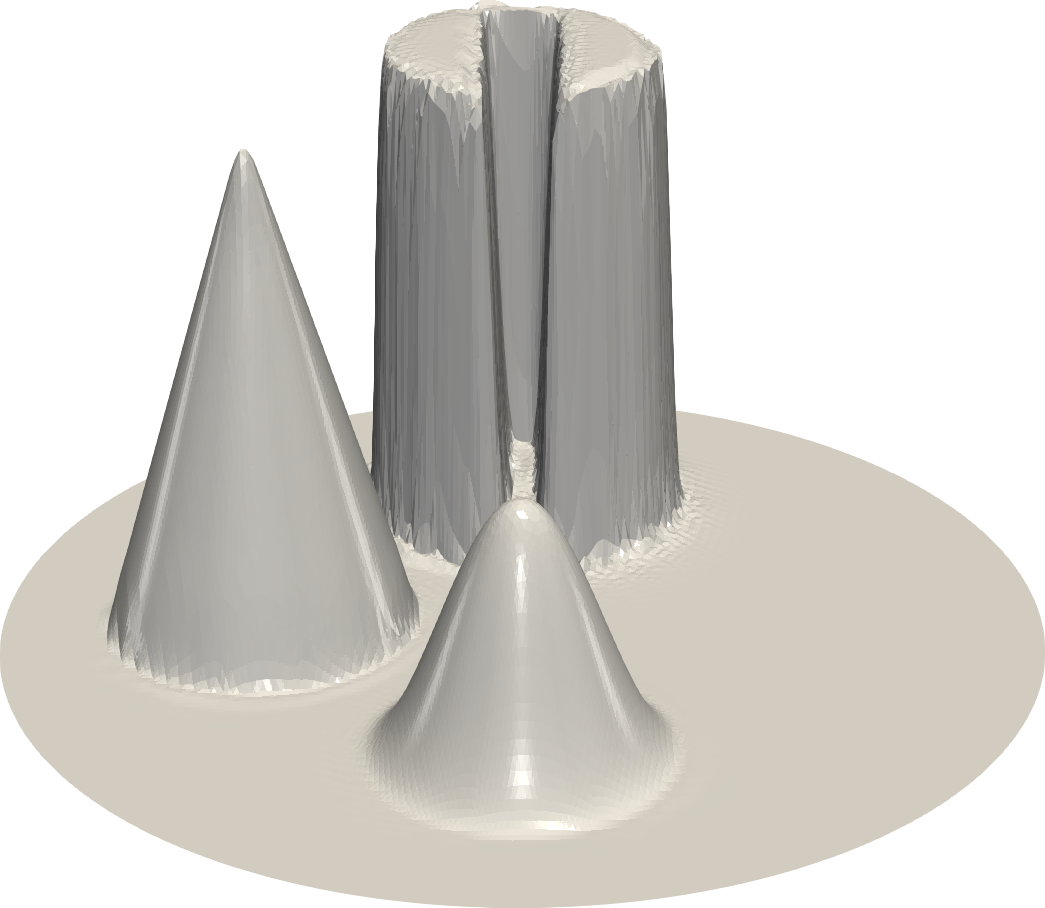}\hfil
\includegraphics[width=0.32\textwidth]{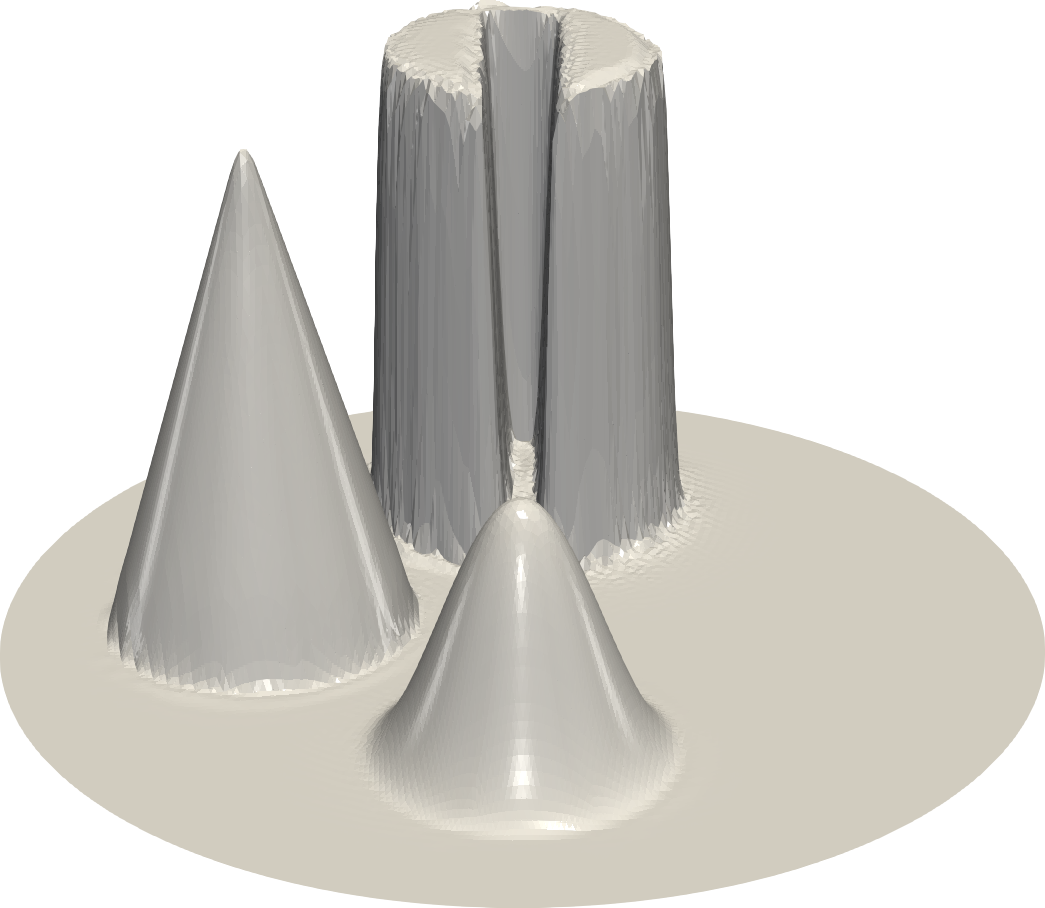}\hfil
\includegraphics[width=0.32\textwidth]{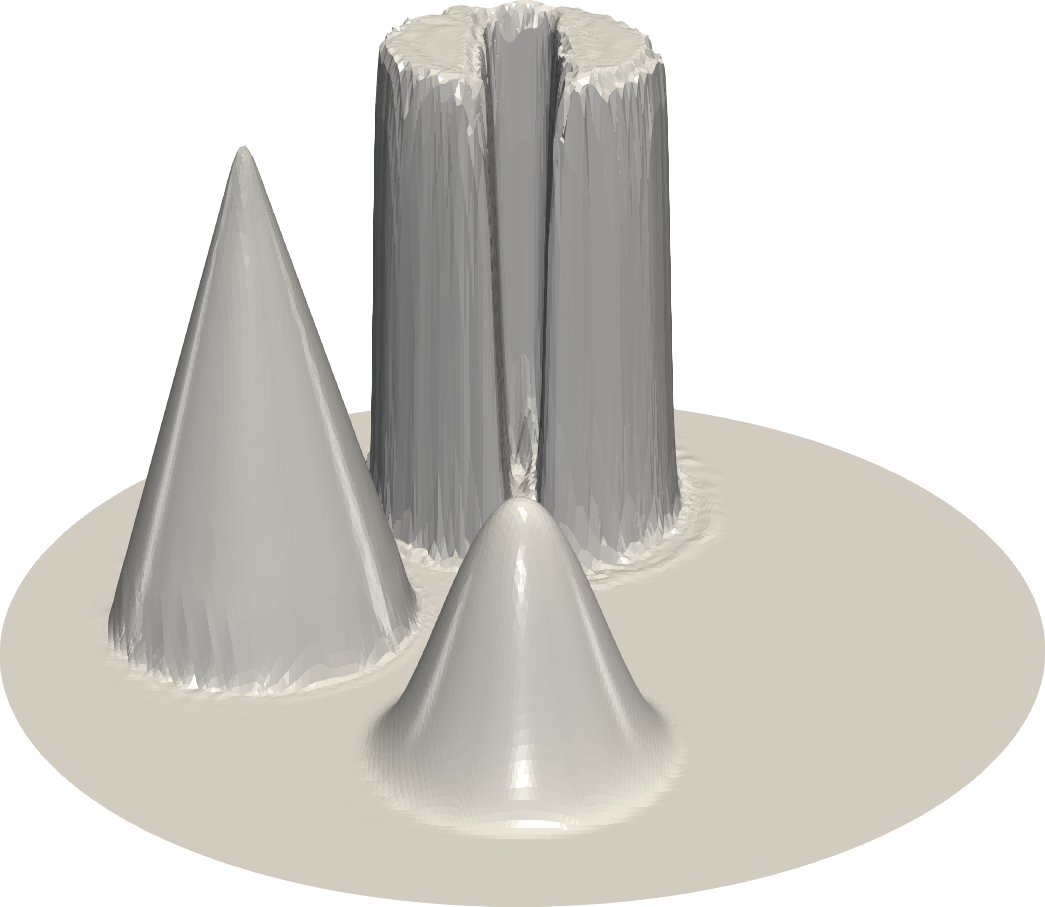}
\caption{Linear rotation problem, $\text{CFL}=0.5$, $t=1$. Left:
  $\polP_1$ approximation, $25\,049$ dofs; Center: $\polP_2$
  approximation, $25\,734$ dofs; Right: $\polP_3$ approximation,
  $27\,091$ dofs}
\label{Fig:three_body}
\end{figure}

\subsection{1D Euler, continuous finite elements}\label{Sec:Sod_and_Lax}

We continue with the compressible Euler equations.  We use the same
method as in \eqref{Sec:three_body_scalar_transport}, but this time
the high-order solution is the Galerkin method augmented with a graph
viscosity involving only the low-order connectivity induced by the
low-order $\polP_1$ mesh $\calT\up{L}$ defined in
\S\ref{Sec:high-degree_FE}.  The high-order graph viscosity applied on
this stencil is estimated by computing a commutator involving the
physical entropy.  The time stepping is done again with the maximally
efficient ERK3 Runge Kutta method; see \cite[Eq. (4.1)]{ErnGu:21+}.
Here again the locally conservative redistribution limiting algorithm
is applied at the end of each Runge-kutta stage to enforce the local
maximum and minimum principle on the density and the local minimum
principle on the physical entropy. The local IDP states $\bu_j^+$ in
each case are estimated using the low-order update with the IDP method
described in \citep{Guermond_Popov_SINUM_2016} and the hierarchical
$\polP_1$ meshes $\calT\up{L}$ described in
\S\ref{Sec:high-degree_FE}.

We solve the Sod and the Lax shocktubes which are two standard Riemann
problems. The computational domain is both cases is $\Dom=(0,1)$.
We show the observed convergence rates in Tables
\ref{Table:Sod}-\ref{Table:Lax} for $\polP_1$, $\polP_2$ and $\polP_3$
continuous finite elements. In each case the CFL number is set to
$0.5$. We observe in each the expected convergence rates.

\begin{table}[ht]\small \centering
  \caption{Convergence tests for Sod shocktube, $\Dom=(0,1)$,
    $t=0.225$. From left to right $\mathbb{P}_1$, $\mathbb{P}_2$,
    $\mathbb{P}_3$. $\text{CFL}=0.5$.}
  \label{Table:Sod}
  \begin{tabular}{c|c|c}
    \multicolumn{3}{c}{$\polP_1$}\\\hline
I &  $\delta_1$ & rate \\ \hline
    121 & 5.46E-02  &  -- \\
    241 & 2.46E-02  & 1.16\\
    481 & 1.43E-02  & 0.78\\
    961 & 6.71E-03  & 1.10\\
   1921 & 3.24E-03  & 1.05\\
   3841 & 1.97E-03  & 0.72
  \end{tabular}
  \begin{tabular}{c|c|c}
    \multicolumn{3}{c}{$\polP_2$}\\\hline
I &  $\delta_1$ & rate \\ \hline
    121 & 4.87E-02  &  -- \\
    241 & 1.95E-02  & 1.33\\
    481 & 1.27E-02  & 0.62\\
    961 & 5.66E-03  & 1.16\\
   1921 & 3.27E-03  & 0.79\\
   3841 & 2.20E-03  & 0.57
   \end{tabular}
  \begin{tabular}{c|c|c}
    \multicolumn{3}{c}{$\polP_3$}\\\hline
I &  $\delta_1$ & rate \\ \hline  
    121 & 8.80E-02  &  -- \\
    241 & 4.25E-02  & 1.06\\
    481 & 2.23E-02  & 0.93\\
    961 & 1.18E-02  & 0.92\\
   1921 & 6.39E-03  & 0.89\\
   3841 & 3.35E-03  & 0.93
   \end{tabular}
\end{table}

\begin{table}[ht]\small \centering
  \caption{ Convergence tests for Lax shocktube, $\Dom=(0,1)$, $t=0.225$. From left to right
    $\mathbb{P}_1$, $\mathbb{P}_2$, $\mathbb{P}_3$. $\text{CFL}=0.5$.}
  \label{Table:Lax}
  \begin{tabular}{c|c|c}
     \multicolumn{3}{c}{$\polP_1$}\\\hline
I &  $\delta_1$ & rate \\ \hline
    121 & 9.02E-02  &  -- \\
    241 & 4.25E-02  & 1.09\\
    481 & 2.14E-02  & 1.00\\
    961 & 1.30E-02  & 0.72\\
   1921 & 9.62E-03  & 0.43\\
   3841 & 5.91E-03  & 0.70
    \end{tabular}
  \begin{tabular}{c|c|c|c|c}
     \multicolumn{3}{c}{$\polP_2$}\\\hline
I &  $\delta_1$ & rate \\  \hline
    121 & 1.10E-01  &  -- \\
    241 & 5.79E-02  & 0.93\\
    481 & 3.14E-02  & 0.88\\
    961 & 2.19E-02  & 0.52\\
   1921 & 1.48E-02  & 0.57\\
   3841 & 8.51E-03  & 0.79
  \end{tabular}
  \begin{tabular}{c|c|c|c|c}
     \multicolumn{3}{c}{$\polP_3$}\\\hline
I &  $\delta_1$ & rate \\  \hline
    121 & 1.30E-01  &  -- \\
    241 & 6.59E-02  & 0.98\\
    481 & 3.82E-02  & 0.79\\
    961 & 2.20E-02  & 0.80\\
   1921 & 1.68E-02  & 0.39\\
   3841 & 1.09E-02  & 0.63
\end{tabular}
\end{table}


  


\subsection{2D Euler, FE} \label{Sec:Sod_and_Lax}

We finish with the Mach 3 step problem. The method (and code) is
exactly the same as that described in \S\ref{Sec:Sod_and_Lax}.  The
setting can be found, for instance, in
\citep[\S5.7]{Guermond_Nazarov_Popov_Tomas_2018} (and \S5.8 of the
supplementary file of \citep{Guermond_Nazarov_Popov_Tomas_2018}).  The
solution to this problem is computed up to $t=4$ using continuous
$\polP_1$, $\polP_2$, and $\polP_3$ finite elements. We show a
Schlieren-type snapshot of the density at $t= 4$ in
Figure~\ref{Fig:Wind_tunnel_Mach3_t4}. We observe the expected
features, including the Kelvin-Helmholtz instability of the contact
discontinuity emerging from the first triple point. The striations are
exaggerated by the Schlieren rendition of the density field.

\begin{figure}[htb]
  \begin{subfigure}{0.5\textwidth}
    \includegraphics[width=0.99\textwidth]{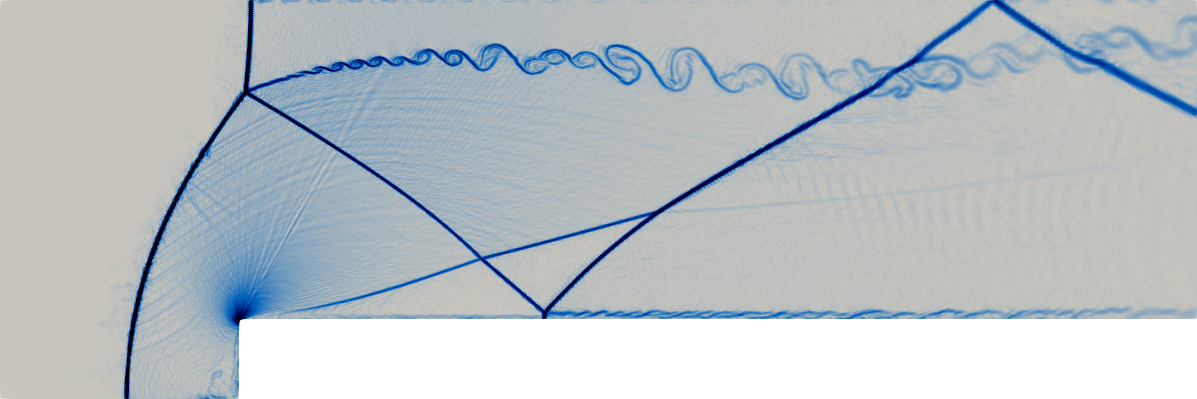}%
  \end{subfigure}%
  \begin{subfigure}{0.5\textwidth}
    \includegraphics[width=0.99\textwidth]{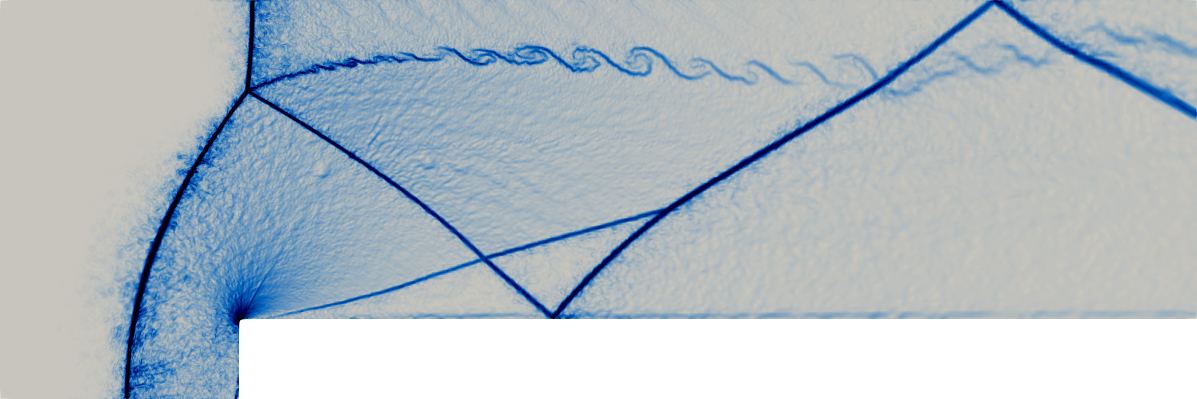}%
  \end{subfigure}\\
 \centering  \begin{subfigure}{0.5\textwidth}
   \includegraphics[width=0.99\textwidth]{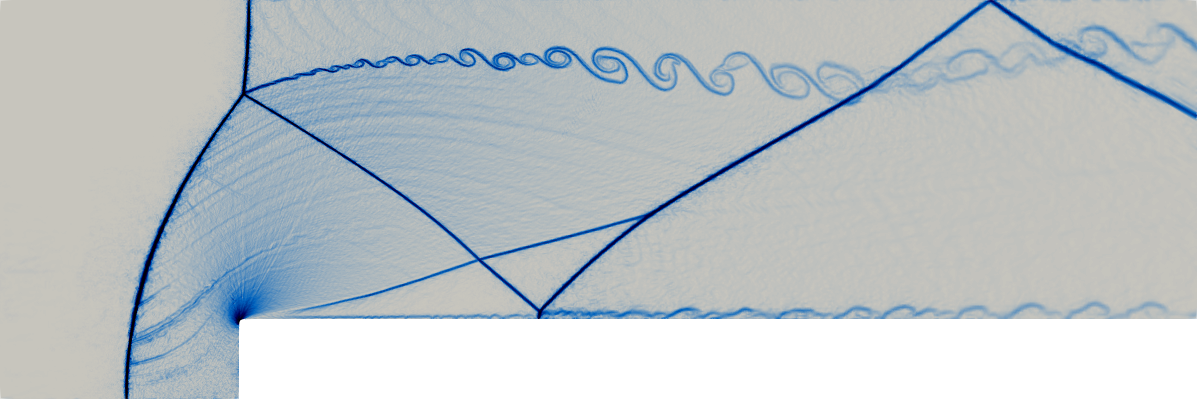}
    \end{subfigure}
\caption{Mach 3 step, $t=4$, $CFL=0.5$. Left to right and top to
  bottom: $\polP_1$ $207\,340$ dofs, $\polP_2$ $136\,487$, $\polP_3$
  $306538$ dofs.}
\label{Fig:Wind_tunnel_Mach3_t4}
\end{figure}

\section*{Acknowledgments} The author has been made aware by M.~Shashkov of the existence of
``repair'' methods during the North American High-Order Methods
Conference held in Santa Fe, NM, June 1--3 2026.  After reading a
first draft of this paper, E.~Tovar also made the author aware of the
sweeping techniques by
\citep{Liu_Cheng_Shu_JSC_2017,Griffin_Shu_RMathSci_2026}. The input
from these two colleagues is greatly appreciated.

\section*{Data availability statements} No data was used for this paper.

\bibliographystyle{abbrvnat}
\bibliography{ref}

\end{document}

%% file: FIGS/submesh_quad_high_order.pdf_t
\begin{picture}(0,0)%
\includegraphics{FIGS/submesh_quad_high_order.pdf}%
\end{picture}%
\setlength{\unitlength}{4144sp}%
\begin{picture}(6946,5369)(-322,-19958)
\put(901,-15811){\makebox(0,0)[lb]{\smash{\fontsize{24}{28.8}\usefont{T1}{ptm}{m}{n}{\color[rgb]{0,0,0}$K$}%
}}}
\put(4186,-15811){\makebox(0,0)[lb]{\smash{\fontsize{24}{28.8}\usefont{T1}{ptm}{m}{n}{\color[rgb]{0,0,0}$K'$}%
}}}
\put(2972,-16346){\makebox(0,0)[lb]{\smash{\fontsize{28}{33.6}\usefont{T1}{ptm}{m}{n}{\color[rgb]{0,0,0}$i$}%
}}}
\end{picture}%

%% file: FIGS/submesh_high_order.pdf_t
\begin{picture}(0,0)%
\includegraphics{FIGS/submesh_high_order.pdf}%
\end{picture}%
\setlength{\unitlength}{4144sp}%
\begin{picture}(7171,7169)(-322,-13883)
\put(2521,-7891){\makebox(0,0)[lb]{\smash{\fontsize{24}{28.8}\usefont{T1}{ptm}{m}{n}{\color[rgb]{0,0,0}$K$}%
}}}
\put(4456,-7981){\makebox(0,0)[lb]{\smash{\fontsize{24}{28.8}\usefont{T1}{ptm}{m}{n}{\color[rgb]{0,0,0}$K'$}%
}}}
\put(3691,-9376){\makebox(0,0)[lb]{\smash{\fontsize{24}{28.8}\usefont{T1}{ptm}{m}{n}{\color[rgb]{0,0,0}$i$}%
}}}
\end{picture}%

%% file: FIGS/submesh_dg_triangles_high_order.pdf_t
\begin{picture}(0,0)%
\includegraphics{FIGS/submesh_dg_triangles_high_order.pdf}%
\end{picture}%
\setlength{\unitlength}{4144sp}%
\begin{picture}(6999,6999)(8539,-13753)
\put(13276,-7981){\makebox(0,0)[lb]{\smash{\fontsize{24}{28.8}\usefont{T1}{ptm}{m}{n}{\color[rgb]{0,0,0}$K'$}%
}}}
\put(9556,-10291){\makebox(0,0)[lb]{\smash{\fontsize{24}{28.8}\usefont{T1}{ptm}{m}{n}{\color[rgb]{0,0,0}$K''$}%
}}}
\put(11776,-8997){\makebox(0,0)[lb]{\smash{\fontsize{24}{28.8}\usefont{T1}{ptm}{m}{n}{\color[rgb]{0,0,0}$i$}%
}}}
\put(11206,-7756){\makebox(0,0)[lb]{\smash{\fontsize{24}{28.8}\usefont{T1}{ptm}{m}{n}{\color[rgb]{0,0,0}$K$}%
}}}
\end{picture}%

%% file: FIGS/submesh_submesh_quad_high_order.pdf_t
\begin{picture}(0,0)%
\includegraphics{FIGS/submesh_submesh_quad_high_order.pdf}%
\end{picture}%
\setlength{\unitlength}{4144sp}%
\begin{picture}(6946,5369)(-322,-19958)
\put(3024,-16376){\makebox(0,0)[lb]{\smash{\fontsize{28}{33.6}\usefont{T1}{ptm}{m}{n}{\color[rgb]{0,0,0}$i$}%
}}}
\put(901,-15811){\makebox(0,0)[lb]{\smash{\fontsize{24}{28.8}\usefont{T1}{ptm}{m}{n}{\color[rgb]{0,0,0}$\calT\upL(i)$}%
}}}
\end{picture}%

%% file: FIGS/submesh_submesh_triangles_high_order.pdf_t
\begin{picture}(0,0)%
\includegraphics{FIGS/submesh_submesh_triangles_high_order.pdf}%
\end{picture}%
\setlength{\unitlength}{4144sp}%
\begin{picture}(7171,7169)(15203,-13883)
\put(19081,-9511){\makebox(0,0)[lb]{\smash{\fontsize{24}{28.8}\usefont{T1}{ptm}{m}{n}{\color[rgb]{0,0,0}$i$}%
}}}
\put(17956,-7801){\makebox(0,0)[lb]{\smash{\fontsize{24}{28.8}\usefont{T1}{ptm}{m}{n}{\color[rgb]{0,0,0}$\calT\upL(i)$}%
}}}
\end{picture}%

%% file: FIGS/submesh_submesh_dg_triangles_high_order.pdf_t
\begin{picture}(0,0)%
\includegraphics{FIGS/submesh_submesh_dg_triangles_high_order.pdf}%
\end{picture}%
\setlength{\unitlength}{4144sp}%
\begin{picture}(7171,7169)(15203,-13883)
\put(19081,-9511){\makebox(0,0)[lb]{\smash{\fontsize{24}{28.8}\usefont{T1}{ptm}{m}{n}{\color[rgb]{0,0,0}$i$}%
}}}
\put(17821,-8971){\makebox(0,0)[lb]{\smash{\fontsize{24}{28.8}\usefont{T1}{ptm}{m}{n}{\color[rgb]{0,0,0}$\calT\upL(i)$}%
}}}
\end{picture}%